\documentclass{amsart}

\usepackage{amssymb,amsmath,amsthm, amsfonts}
\usepackage[all]{xy}
\usepackage{upgreek}
\usepackage{stmaryrd}
\usepackage[hidelinks]{hyperref}
\usepackage{rotating}
\usepackage{tikz}
\usepackage{enumitem,kantlipsum}

\newcommand{\bk}{\Bbbk}
\newcommand{\Z}{\mathbb{Z}}

\newcommand{\C}{\mathbb{C}}

\newcommand{\fg}{\mathfrak{g}}

\newcommand{\uenv}{\mathcal{U}}
\newcommand{\uenvt}{\widehat{\mathcal{U}}}

\newcommand{\irr}{\mathsf{L}}
\newcommand{\bv}{\mathsf{Z}}
\newcommand{\bm}{\mathsf{M}}

\newcommand{\Weyl}{\mathsf{V}}
\newcommand{\Inj}{\mathsf{Q}}
\newcommand{\pdist}{\mathsf{D}}

\newcommand{\rad}{\text{rad}}
\newcommand{\soc}{\text{soc}}
\newcommand{\mycap}{\text{cap}}
\newcommand{\socl}{\overline{\text{soc}}}
\newcommand{\radl}{\overline{\text{rad}}}
\newcommand{\olF}{\overline{\text{F}}}
\newcommand{\nolF}{\text{F}}

\newcommand{\Dist}{\mathrm{Dist}}

\newcommand{\bX}{\mathbf{X}}

\newcommand{\cC}{\mathcal{C}}

\newcommand{\Rep}{\mathsf{Rep}}

\DeclareMathOperator{\Hom}{Hom}
\DeclareMathOperator{\Ext}{Ext}

\DeclareMathOperator{\Ind}{ind}
\DeclareMathOperator{\coInd}{coind}

\newcommand{\mcK}{\mathcal{K}}

\newcommand{\wh}[1]{\widehat{#1}}

\newcommand{\la}{\langle}
\newcommand{\ra}{\rangle}

\makeatletter
\def\@secnumfont{\bfseries}  
\def\lotimes{\@ifnextchar_{\@lotimessub}{\@lotimesnosub}}
\def\@lotimessub_#1{\mathchoice{\mathbin{\mathop{\otimes}^L}_{#1}}%
  {\otimes^L_{#1}}{\otimes^L_{#1}}{\otimes^L_{#1}}}
\def\@lotimesnosub{\mathbin{\mathop{\otimes}^L}}
\makeatother

\newcommand{\Wp}{W_{\mathrm{p}}}
\newcommand{\Wpext}{W^{\mathrm{ext}}_{\mathrm{p}}}

\newcommand{\Irr}{\mathrm{Irr}}

\numberwithin{equation}{section}
\newtheorem{thm}{Theorem}[section]
\newtheorem{lem}[thm]{Lemma}
\newtheorem{prop}[thm]{Proposition}
\newtheorem{cor}[thm]{Corollary}
\newtheorem{conj}[thm]{Conjecture}

\theoremstyle{definition}

\theoremstyle{remark}
\newtheorem{rmk}[thm]{Remark}

\title[Explicit calculations in an infinitesimal singular block of $SL_N$]{Explicit calculations in an infinitesimal singular block of $SL_{N}$}

 \author{William Hardesty}
 \address{School of Mathematics and Statistics F07\\
   University of Sydney NSW 2006\\
   Australia
}
 \email{william.hardesty@sydney.edu.au}

\begin{document}

\maketitle
\begin{abstract}
Let $G= SL_{n+1}$ be defined over an algebraically closed field of characteristic $p > 2$. 
For each $n \geq 1$ there exists a singular block in the category of $G_1$-modules which
contains precisely $n+1$ irreducible modules. We are interested in the ``lift" of this block to the category of $G_1T$-modules. Imposing only mild assumptions on $p$, 
we will perform a number of calculations in this setting, including a complete 
determination of the Loewy series for the baby Verma modules and all possible extensions between the irreducible modules.
In the case where $p$ is extremely large, we will also explicitly compute the Loewy series for the indecomposable projective modules.
\end{abstract}


\section{Introduction}\label{sec:introduction}
Let $G= SL_{n+1}$ with $n\geq 1$ be defined over an algebraically closed field $\bk$ of characteristic $p > 2$.
We will also assume that $p$ is a \emph{very good} prime for $G$ (i.e. $p\nmid n+1$ for $SL_{n+1}$).
  The setting of this paper is centered around the
representation theory of the subgroups $G_1$ and $G_1T$, where $G_1$ is the Frobenius kernel and $T\subseteq G$ is the subgroup 
of diagonal matrices. More specifically, we will focus on an important class of representations known as the 
\emph{baby Verma modules}. These are certain finite-dimensional representations which are highly analogous to the ``classical" Verma modules for complex semi-simple Lie algebras. 
We are also interested in an important invariant known as the \emph{Loewy series} 
(or the \emph{radical series}) of a module. The invariant provides a significant amount of information on the submodule structure of a representation, but is often impractical to compute. 

Determining the Loewy series of baby Verma modules for $G_1T$ has been a
particularly important topic in this history of representation theory for algebraic groups.  Major progress was first made in the 1990s, when 
 Andersen, Jantzen, and Soergel demonstrated that for $p \gg 0$,  the
 Loewy series of any baby Verma module whose highest weight is $p$-regular, can be expressed in terms of the \emph{periodic Kazhdan--Lusztig basis} (see \cite{ajs}). 
Recently, Abe and Kaneda in \cite{nm},  building on a 2010 result by Riche (\cite{riche}), were able to extend these results to include baby Verma modules of any highest weight.
Their methods
 depend on the validity of Lusztig's character formula, as well as some additional assumptions from \cite{riche}
 (see \cite[Theorem, p. 2]{nm}).   
Unfortunately, by the well-known result of Williamson (\cite{williamson}), these character formulas are 
often only valid for extremely large primes. It is also worth mentioning that the periodic Kazhdan--Lusztig basis is highly difficult to compute. So even in the case where $p$ is sufficiently large, it remains challenging to obtain 
precise information with these methods.
 
In this article, we take a more specialized approach and restrict ourselves to a particular subcategory of 
$G_1T$-modules which is related to an interesting singular block of $G_1$
(see \eqref{eqn:lift-of-block} for an explicit definition). 
The most significant result is {\bf Theorem~\ref{thm:g1t-radical-layers}} which gives precise 
formulas for the Loewy series of every baby Verma module in this subcategory. 
A key consequence is that these baby Verma modules are \emph{rigid}
(see {\bf Proposition~\ref{prop:rigidity}}). 
Our formulas are independent of $p$, and thus agree with the results of \cite{nm}, but our
techniques differ considerably from {\em loc. cit.}
Another major result is {\bf Theorem~\ref{thm:ext1-thm}}, which gives a complete 
determination of the extensions between the irreducible objects.
Amazingly, we are able to prove all of these results under the mild assumption that $p$ is very good. 

Finally in \S\ref{sec:radical-layers}, we impose an additional condition on $p$ which is known to hold when 
$p$ is extremely large (see Remark~\ref{rmk:large-p}). In this case, our methods also lead to
an explicit description of the Loewy series for every indecomposable projective module in {\bf Theorem~\ref{thm:singular-reciprocity}}. 
The strategy of our proof involves combining our baby Verma calculations 
with the results of \cite{nm} and by adapting the techniques from \cite{ak1989} over to our setting.

 To the author's best knowledge, this paper
 gives the first known example of an infinite family of non-trivial
singular blocks for $G_1T$ with $G = SL_{n+1}$ (as $n\geq 1$ varies), in which \cite[Theorem, p. 2]{nm} 
holds for ``reasonable" primes.\footnote{Similar results for some low-rank groups already exist in the literature (cf. \cite{towers}, \cite{xi1999}, and \cite{xi2009}).} 
By contrast, recall from \cite{williamson}
that if we consider the corresponding family of principal blocks and let $p(n)$ be the minimal prime for each $n$ such that
\cite[Theorem, p. 2]{nm} holds, then the growth rate of the function $p(n)$ is actually \emph{non-polynomial} (see \cite{fiebig} for an explicit 
upper bound to $p(n)$).

   As a consequence, we can see that even though the principal block is often poorly behaved for 
   smaller primes, there can still exist interesting singular blocks which are well behaved under milder assumptions on the characteristic. 
   These kinds of blocks have also been the subject of a recent preprint (\cite{nz}), where the authors studied a family of singular $G_1$-blocks occurring in a categorified $\mathfrak{sl}_2(\C)$ 
    representation. 
   
  {\bf The case of general linear groups:}
 It is important to note that the results of this paper can be adapted to the case where $G = GL_{n+1}$ 
   (see Remark~\ref{rmk:type-a-gen}). Moreover, since all primes are very good for $GL_{n+1}$, it should also be straightforward to check that all of our
   results from \S\ref{sec:notation}--\S\ref{sec:verma-radical-layers} extend to \emph{every} odd $p$ 
    in this situation. 
 

{\bf Acknowledgments:} The author would like to thank V. Nandakumar for providing the motivation for this project and a number of key insights, as well as D. Nakano and J. Humphreys for their useful comments and suggestions. The author also expresses his sincere gratitude to the
 anonymous referee for providing a very thorough report.

\section{Notation and preliminaries}\label{sec:notation}
Maintaining the same assumptions as in the introduction, we let $B \subseteq G$ denote the Borel subgroup consisting of lower triangular matrices.
 The weight lattice of $G$ is given by
 \[
 \bX = \Z^{n+1}/\ \la e_1 + \cdots + e_{n+1} \ra,
 \] 
 where $e_1,\dots, e_{n+1}$ are standard basis vectors. Set $\epsilon_i = \overline{e_i} \in \bX$ for $i=1,\dots, n+1$, and let 
 $\varpi_i = \epsilon_1 + \cdots +\epsilon_i$ for $i=1,\dots, n$ be the fundamental weights.  The root system is given by 
 $${\Phi=\{\epsilon_i - \epsilon_j  \, \mid \, 1\leq i,j \leq n+1, \, i \neq j\}}.$$
We take
 $${\Phi^+= \{\epsilon_i - \epsilon_j \, \mid \, 1\leq i<j \leq n+1\}}$$
to be the set of positive roots with basis
 $${S = \{\epsilon_i - \epsilon_{i+1} \,\mid\, 1 \leq i \leq n\}}.$$ 
We will set $\alpha_i = \epsilon_i-\epsilon_{i+1}$ for $i = 1,\dots, n$. 

 The Weyl group is $W = \Sigma_{n+1}$ (the group of permutations on $n+1$ letters), and its action on $\bX$ is induced by the natural action of permuting coordinates. The longest element $w_0 \in W$ is the permutation given by
\begin{equation}\label{eqn:long-element}
w_0: i \mapsto n+2-i
\end{equation}
for $i=1,\dots, n+1$.

 The affine Weyl group is given by
 \[
 \Wp = W \ltimes p\Z\Phi,
 \]
where $\Z\Phi$ acts on $\bX$ by translations (and hence $p\Z\Phi$ acts by translations of elements in $p\Z\Phi \subset \Z\Phi$). We similarly define the 
extended affine Weyl group 
\[
\Wpext = W \ltimes p\bX.
\]
As usual, we set $\rho = \frac{1}{2}\sum_{\alpha \in \Phi^+} \alpha$ and define the \emph{dot action}
 of $\Wp$ (or $\Wpext$) on $\bX$ by $w\cdot \lambda = w(\lambda +\rho) - \rho$ for any $w \in \Wp$ (or $\Wpext$) and $\lambda \in \bX$. This extends to an action on $\bX \otimes \mathbb{R}$ and defines a system of \emph{facets} for $\bX \otimes \mathbb{R}$. 

For any group scheme $H$, let $\Rep(H)$ denote the category of finite-dimensional $H$-modules, and let 
$\mcK(H)$ denote the Grothendieck group. For an $H$-module $M$, let $[M] \in \mcK(H)$ denote its class,
and for an $H$-module $N$, take  $[M] \leq [N]$ to mean ${[M:\irr] \leq [N:L]}$ for every irreducible
$H$-module $\irr$ (where $[M:\irr]$ is the multiplicity of $\irr$ in any Jordan--H\"older filtration). In particular, the class of any $H$-module $M$ has the unique expansion 
\begin{equation}\label{eqn:irreducible-expansion}
[M] = \sum_{\irr \in \Irr(H)} [M:\irr][\irr],
\end{equation}
where $\Irr(H)$ denotes the set of (isomorphism classes of) irreducible representations. 

The \emph{radical} of an $H$-module $M$, denoted $\rad\, M$, is defined to be the intersection of
all maximal submodules of $M$.  For $i\geq 0$,  
$\rad^i\,M$ is given by
\[
\begin{aligned}
\rad^0\, M &= M \\
\rad^i\,M &= \rad(\rad^{i-1}\,M) \quad \text{for $i \geq 1$}. 
\end{aligned}
\]
The $i^{th}$-\emph{radical layer} of $M$ is given by
\[
\radl_i\,M = \rad^i\,M/\rad^{i+1}\,M.
\]

Similarly,  let $\soc\, M$ denote the \emph{socle} of $M$, which is the sum of all simple submodules of $M$. For $i\geq 0$,
$\soc^i\, M$ is given by
\[
\begin{aligned}
\soc^0\, M &= 0, \\
\soc^i\,M &= \pi^{-1}(\soc(M/\soc^{i-1}\, M)) \quad \text{with $\pi: M \twoheadrightarrow M/\soc^{i-1}\,M$ for $i \geq 1$}. 
\end{aligned}
\]
The $i^{th}$-\emph{socle layer} is the subquotient
\[
\socl_i\, M = \soc^i\, M/\soc^{i-1}\, M.
\]
We will also set
\[
\mycap^i\, M = M/\rad^i\, M,
\]
for $i\geq 0$. Observe that $\mycap^1\, M = \radl_0(M)$, this is often called the \emph{head} of $M$. 

Suppose $N\subseteq M$ is any submodule, then it is easy to verify that
for $i\geq 0$, 
\begin{equation}\label{eqn:radical-submodule}
\soc^i\, N = (\soc^i\, M) \cap N, \quad \rad^i\, M/N = \frac{\rad^i\, M + N}{N}, \quad
\mycap^i\, M/N = \frac{M}{\rad^i\, M + N}.
\end{equation}
In particular, if $\pi: M \twoheadrightarrow M'$ is a surjection, then 
$\pi(\rad^i\,M) = \rad^i\, M'$.

The \emph{Loewy length} of $M$ is defined to be the smallest 
integer $r\geq 0$ such that $\rad^r(M) = 0$ (or equivalently $\soc^r\, M = M$); we will denote this by $\ell\ell(M)$ (see \cite[II.D.1]{jantzen}). 
For all $0 \leq i \leq \ell\ell(M)$, 
\begin{equation}\label{eqn:socle-cap}
[M] \leq [\soc^i\, M] + [\mycap^{\ell\ell(M)-i}\, M].
\end{equation}
The module $M$ is said to be \emph{rigid} whenever equality holds for all $i$ (see \cite[(4)]{ak1989} or \cite[D.9]{jantzen}). 

Another useful observation is that if $\pi: M \twoheadrightarrow M'$ is surjective and $\ell\ell(M') = r$, then 
there is an induced surjection
\begin{equation}\label{eqn:cap-factor}
\overline{\pi}: \mycap^r\, M \twoheadrightarrow M'. 
\end{equation}
(Equivalently, any surjective map from $M$ to a module of Loewy length $r$ factors through $\mycap^r\, M$.)

\begin{rmk} 
If $M$ is a $G_1T$-module, then 
\begin{equation*}\label{eqn:restrict-G1}
\rad^i(M|_{G_1}) \cong (\rad^i\, M)|_{G_1}, \quad \soc^i(M|_{G_1}) \cong (\soc^i\, M)|_{G_1}.
\end{equation*}
So in particular, $\ell\ell(M) = \ell\ell(M|_{G_1})$. 
\end{rmk}

The irreducible representations of $G_1$ are indexed by 
the set of $p$-restricted weights $\bX_1 = \{ \sum a_i\varpi_i \mid 0\leq a_i < p\}$, and will be denoted by $\irr(\lambda)$ for $\lambda \in \bX_1$. 
The irreducible representations of $G_1T$ will be denoted by $\wh{\irr}(\lambda)$ for all $\lambda \in \bX$, where we recall that if we write
$\lambda = \mu + p\nu$ for some $\mu \in \bX_1$ and $\nu \in \bX$, then 
\begin{equation}\label{eqn:steinberg-tensor}
\wh{\irr}(\lambda) \cong \irr(\mu) \otimes p\nu. 
\end{equation}

Let $B^+$ denote the opposite Borel subgroup consisting of upper triangular matrices. For any $\lambda \in \bX$, we define the baby Verma and dual baby Verma modules respectively by
\[
\widehat{\bv}(\lambda) = \coInd_{B_1^+T}^{G_1T}\lambda, \quad \wh{\bv}'(\lambda) = \Ind_{B_1T}^{G_1T}\lambda
\]
(see \cite[II.9]{jantzen} for an overview). We also let ${\wh{\Inj}(\lambda) }$ denote the projective cover of 
${\wh{\irr}(\lambda)}$ (see \cite[II.11]{jantzen}). The corresponding $G_1$-modules are given by
\[
\bv(\lambda) = \coInd_{B_1^+}^{G_1}\lambda , \quad \bv'(\lambda) = \Ind_{B_1}^{G_1}\lambda, \quad \Inj(\lambda)
\]
for any $\lambda \in \bX_1$.

Let $\tau: G \rightarrow G$ be the anti-automorphism given by matrix transposition. In this specific case, it is obvious that $\tau$ fixes $T$ and interchanges $B$ and $B^+$ 
(see \cite[Corollary II.1.16]{jantzen} for the more general statement). It is also well-known that $\tau$ commutes with the Frobenius map, and hence preserves 
$G_1$ and $G_1T$. 

If $H \leq G$ is any subgroup scheme preserved by $\tau$ and $M$ is any $H$-module, then the 
\emph{twist}
${ }^{\tau}M$ is an $H$-module called the \emph{$\tau$-dual} of $M$ (see \cite[I.2.15]{jantzen}). We also have
\begin{equation}\label{eqn:dual-filtrations}
{ }^{\tau}(\mycap^i\, M) \cong \soc^i\, ({ }^{\tau}M), \quad { }^{\tau}(\radl_i\, M) \cong \socl_{i+1}({ }^{\tau}M)
\end{equation}
for $i\geq 0$, and in particular, ${\ell\ell({ }^{\tau}M) = \ell\ell(M)}$ for any $H$-module $M$. 

If $H=G_1T$, then by \cite[II.9.3(5), II.9.6(13), II.11.5(5)]{jantzen}, 
\begin{equation}\label{eqn:duality-G1T}
{ }^{\tau}\wh{\bv}(\lambda) \cong \wh{\bv}'(\lambda), \quad { }^{\tau}\wh{\irr}(\lambda) \cong \wh{\irr}(\lambda),
\quad { }^{\tau}\wh{\Inj}(\lambda) \cong \wh{\Inj}(\lambda)
\end{equation}
for any $\lambda \in \bX$. Consequently, 
\begin{equation}\label{eqn:socle-radical-duality}
\socl_{i+1}\, \wh{\bv}'(\lambda) \cong \radl_{i}\wh{\bv}(\lambda), \quad \socl_{i+1}\, \wh{\Inj}(\lambda) \cong \radl_{i}\wh{\Inj}(\lambda)
\end{equation}
for $i \geq 0$. (Similar statements hold for $H=G_1$.)

For any $\lambda \in \bX$, let $\wh{\cC}(\lambda)$ denote the block of $\Rep(G_1T)$ whose irreducible objects are given by $\wh{\irr}(w\cdot \lambda)$ for $w \in \Wp$ (cf. \cite[II.9.22]{jantzen}).
Let $\overline{C} \subset \bX \otimes \mathbb{R}$ denote the closure of the bottom alcove $C$, and recall from \cite[II.6.2(5)]{jantzen} that $\overline{C}\cap \bX$ is a fundamental domain for the dot action of 
$\Wp$ on $\bX$. Thus, since
$\wh{\cC}(\lambda) = \wh{\cC}(w\cdot \lambda)$ for any
$w \in \Wp$,  it follows that $\overline{C}\cap \bX$ forms an indexing set for the blocks of $\Rep(G_1T)$. 
For any facet $F \subset \overline{C}$, and any $\lambda, \mu \in F\cap \bX$, the $G_1T$-translation functors $T^{\lambda}_{\mu}(-)$ and 
$T^{\mu}_{\lambda}(-)$ are mutually inverse and induce an equivalence $\wh{\cC}(\lambda) \cong \wh{\cC}(\mu)$ (see \cite[II.9.4]{jantzen}).  

For any $\lambda \in \bX_1$, we similarly let $\cC(\lambda)$ denote the block of $\Rep(G_1)$ whose irreducible
objects are given by $\irr(\mu)$ for $\mu \in (\Wpext\cdot \lambda) \cap \bX_1$ (see \cite[II.9.22(1)]{jantzen}). We also let $\tilde{\cC}(\lambda)$ denote the subcategory
of $\Rep(G_1T)$ generated by blocks of the form $\wh{\cC}(\lambda')$ where $\lambda' \in (\Wpext\cdot \lambda)\cap \overline{C}$. By \eqref{eqn:steinberg-tensor},
the $\irr(\mu) \otimes p\nu$ for $\mu \in (\Wpext\cdot \lambda) \cap \bX_1$ and $\nu \in \bX$ (or equivalently, the $\wh{\irr}(w\cdot \lambda)$ for $w \in \Wpext$), form the set of isomorphism classes of irreducible objects of $\tilde{\cC}(\mu)$. We will refer to $\tilde{\cC}(\mu)$ as the \emph{lift} of 
$\cC(\mu)$ to $\Rep(G_1T)$.

{\bf The subcategory $\tilde{\cC}(\lambda_0)$:} We will now describe the subcategory 
under consideration in
 this paper. For $i=0,\dots, n$, set
\begin{equation}\label{eqn:mu-formula}
\mu_i = \epsilon_{i+1}-\rho
\end{equation}
Each $\mu_i$ has a unique representative $\lambda_i \in \bX_1$, which is given by
\begin{equation}\label{eqn:weight-formula2}
  \begin{aligned}
    \lambda_i &=   \mu_i + p\rho - p\varpi_{i+1} = \epsilon_{i+1} +(p-1)\rho- p\epsilon_1 - \cdots -p\epsilon_{i+1} \quad \text{for  $0 \leq i \leq n-1$,}\\ 
    \lambda_n &= \mu_n + p\rho = \epsilon_{n+1} +(p-1)\rho -  p\epsilon_1 - \cdots -p\epsilon_{n+1}.
  \end{aligned}
\end{equation}
It is easy to verify that
\[
(\Wpext\cdot \lambda_0) \cap \bX_1 = \{ \lambda_0,\dots, \lambda_n\},
\]
and thus $\cC(\lambda_0)$ is the block of $\Rep(G_1)$ where $\{\irr(\lambda_0), \dots, \irr(\lambda_n)\}$ gives the complete set of isomorphism classes of
irreducibles. The lift $\tilde{\cC}(\lambda_0)$ is the full abelian subcategory of $\Rep(G_1T)$ 
which is generated by the set
\begin{equation}\label{eqn:lift-of-block}
\big\{\wh{\irr}(\lambda_i + p\nu)\,\mid\, i \in \{0,\dots,n\}, \, \nu \in \bX\big\}\subset \Irr(G_1T).
\end{equation}

\begin{rmk}\label{rmk:generalize-blocks}
More generally, for $1 \leq a \leq p-1$ we can define the weights
\begin{equation}\label{eqn:weight-formula}
  \begin{aligned}
    \lambda_0^a &= (a-1)\varpi_1+ (p-1)\varpi_2 + \cdots +(p-1)\varpi_{n}, \\
    \lambda_n^a &= (p-1)\varpi_1 + (p-1)\varpi_2 + \cdots + (p-1)\varpi_{n-1} + (p-a-1)\varpi_n, \\
    \lambda_i ^a&= (p-a-1)\varpi_i + (a-1)\varpi_{i+1} + \sum_{j\not\in \{i,i+1\}}(p-1)\varpi_j \quad \text{for  $1 \leq i \leq n-1$},
  \end{aligned}
\end{equation}
and observe that $\lambda^1_i = \lambda_i$ for all $i$. 
Also, for any fixed $i$ and $\nu \in \bX$, the weights $\lambda^a_i+p\nu$ defined above are all contained in the same facet for any choice of $a$. It then follows from \cite[II.9.22(2), II.9.22(4)]{jantzen} that the translation functors
$T^{\lambda_i+p\nu}_{\lambda^a_i+p\nu}(-)$ and 
$T^{\lambda^a_i +p\nu}_{\lambda_i+p\nu}(-)$ induce equivalences
$
\wh{\cC}(\lambda^a_i + p\nu) \cong \wh{\cC}(\lambda_i+p\nu),
$
and also
$
\tilde{\cC}(\lambda^a_0) \cong \tilde{\cC}(\lambda_0).
$
\end{rmk}


\begin{rmk}\label{rmk:type-a-gen}
All of our results extend to any group $G$ whose derived subgroup is isomorphic to 
$SL_{n+1}$. This includes the case where $G = GL_{n+1}$. 
Another important example is given by taking 
$G$ to be the Levi factor $L_I \subset SL_{n+r+1}$ for any $r \geq 1$, where
$${I = \{\epsilon_1-\epsilon_2,\dots, \epsilon_{n}-\epsilon_{n+1}\}.}$$
We can then let $\cC_I(\lambda_0)$ denote the block of $(L_I)_1$ whose irreducible objects are given by $\irr_I(\lambda_0), \dots, \irr_I(\lambda_n)$.
(Notice that the weights $\epsilon_{1},\dots, \epsilon_{n+1}$ are linearly independent in this situation.)
\end{rmk}

\section{Baby Verma modules for parabolic subgroups}\label{sec:para-baby}
In this section we will review some key properties of baby Verma modules. We will also consider an analogous 
family of modules associated to arbitrary parabolic subgroups. 
All of the results in this section should be applicable to arbitrary (connected) reductive algebraic groups.

Let $P_I^+ \subseteq G$ be a (positive) parabolic with Levi decomposition 
$P^+_I = L_I \ltimes U_I^+$ for $I \subseteq S$ (cf. \cite[II.1.8]{jantzen}). Also, let $W_I \subseteq W$ be the Weyl group of $L_I$. 
 (We will often write $P = P_I$ and $L = L_I$ when the subset $I$ is implicit.)
Let $\wh{\irr}_I(\lambda)$ and $\wh{\bv}_I(\lambda)$ denote the irreducible and 
baby Verma modules modules for $L_1T$ (or $P_1^+T$ by inflation) respectively.
For any $\lambda \in \bX$, we define the $G_1T$-module
\begin{equation}
\widehat{\bm}_I(\lambda) = \coInd_{P_1^+T}^{G_1T}\, \widehat{\irr}_{I}(\lambda).
\end{equation}
If $\lambda \in \bX_1$, then 
$\bm_I(\lambda) = \coInd_{P_1^+}^{G_1}\, \irr_{I}(\lambda)$
 is the restriction of $\wh{\bm}_I(\lambda)$ to $G_1$.
\begin{rmk}\label{rmk:dual-para-baby}
The duals of these modules are obtained by applying $\Ind_{P_1T}^{G_1T}(-)$ to $\wh{\irr}_I(\lambda)$, and are  denoted by $\widehat{\bm}'_I(\lambda)$.
\end{rmk}


We now recall an alternative description of the baby Verma modules. 
Let $\uenv = \Dist(G_1) \cong U^{[p]}(\fg)$ and $\uenvt = \Dist(G_1T)$ be two subalgebras of $\Dist(G)$,  
where $U^{[p]}(\fg)$ is the restricted universal enveloping algebra of $G$. 
Following \cite[II.1.11]{jantzen}, let ${\{X_{\alpha}\}_{\alpha \in \Phi}}$ and ${\{H_i\}_{i=1,\dots,n}}$ denote 
the Chevalley basis of $\fg_{\Z}$.

Now by \cite[II.1.12]{jantzen}, 
$\uenv \leq \Dist(G)$ is the subalgebra generated by $X_{\alpha_i}$ and 
$H_i$ for $i=1,\dots, n$. Similarly, $\uenvt \leq \Dist(G)$ is the subalgebra generated by 
${X_{\alpha_i}}$ and ${H_i\choose m}$ for $i=1,\dots, n$ and $m\geq 1$, where 
\[
{H_i\choose m} = \frac{H_i(H_i-1)\cdots (H_i-m+1)}{m!}.
\]

For any $\lambda = a_1\varpi_1 + \cdots + a_n \varpi_n \in \bX$, let 
$I_{\lambda} \unlhd \uenv$ be the left ideal generated by $H_i - a_i\cdot 1$ and $X_{\alpha_i}$ for $i=1,\dots, n$. There is a well-known isomorphism
\[
\bv(\lambda) \cong \uenv/I_{\lambda}.
\]
Likewise, let $\widehat{I}_{\lambda} \unlhd \uenvt$ be the left ideal generated by 
$\left[{H_i \choose m} - {a_i \choose m }\right] \cdot 1$ and $X_{\alpha_i}$ for $i=1,\dots, n$ and $m \geq 1$. We also have an isomorphism
\[
\widehat{\bv}(\lambda) \cong \uenvt/\widehat{I}_{\lambda}.
\]
By \cite[II.9.2]{jantzen}, the elements
\begin{equation}\label{eqn:weight-basis-dist}
\Pi_{\alpha \in \Phi^+} X_{-\alpha}^{n(\alpha)} \cdot \overline{1}
\end{equation}
for $0 \leq n(\alpha) < p$ give a basis of weight vectors for $\wh{\bv}(\lambda)$. (The weight corresponding 
to the element $\Pi_{\alpha \in \Phi^+} X_{-\alpha}^{n(\alpha)}$ is given by 
$\lambda - \left(\sum_{\alpha \in \Phi^+} n(\alpha)\alpha\right)$.)

To obtain a similar description of the $\wh{\bm}_I(\lambda)$, let us first set
\[
\pdist_I = \Dist((U_I)_1)\cong U^{[p]}(\mathfrak{u}_I),
\]
where $\Dist((U_I)_1)$ is the distribution algebra of $(U_I)_1$ and $U^{[p]}(\mathfrak{u}_I)$
is the restricted enveloping algebra of $\mathfrak{u}_I$. We also note that $\pdist_I$ has the natural structure of a $T$-module arising from the conjugation action of $T$ on $(U_I)_1$. 
Following the notation of \cite[II.1]{jantzen}, 
the elements
\begin{equation}\label{eqn:weight-basis-dist}
\Pi_{\alpha \in \Phi^+ \backslash \Phi_I^+} X_{-\alpha}^{n(\alpha)}
\end{equation}
for $0 \leq n(\alpha) < p$ give a basis of weight vectors
for $\pdist_I$ with respect to this action.  In particular, the lowest weight of 
$\pdist_I$ is given by 
\[
\mu_I = \sum_{\alpha \in \Phi^+\backslash \Phi_I^+} -(p-1)\alpha.
\]

For any $L_1$-module $M$ (regarded as a $P_1^+$-module),  the arguments from \cite[II.3.6]{jantzen} give a vector space isomorphism
\begin{equation}\label{eqn:parabolic-induction-tensor}
\coInd_{P_1^+}^{G_1}\, M \cong \pdist_I \otimes M,
\end{equation}
which is compatible with the natural $(U_I)_1$ and $L_1$ module structures on $\pdist_I \otimes M$. Specifically, $(U_I)_1$ acts on $\pdist_I$ by the regular representation and acts on $M$ trivially, while $L_1$ acts on $\pdist_I$ by the adjoint action (induced from conjugation) and on $M$ by left multiplication. Furthermore, these actions are compatible with the conjugation action of $L_1$ on $(U_I)_1$, and thus \eqref{eqn:parabolic-induction-tensor} is actually an isomorphism of $P_1$-modules. 

Likewise, if $M$ is a $L_1T$-module (regarded as a $P_1^+T$-module), then the
 arguments from 
\cite[II.9.2]{jantzen} give a $P_1T$-module isomorphism
\begin{equation}\label{eqn:parabolicT-induction-tensor}
\coInd_{P_1^+T}^{G_1T}\, M \cong \pdist_I \otimes M.
\end{equation}

\begin{lem}\label{lem:parabolic-vermas}
Let $\lambda \in \bX$.
\begin{enumerate}
\item $\wh{\bm}_I(\lambda)$ is a quotient of $\wh{\bv}(\lambda)$.
\item The lowest weight of  $\widehat{\bm}_I(\lambda)$ is given by $\mu_I + w_I(\lambda)$, where $w_I \in W$ is the longest element of 
   $W_I$. 
\end{enumerate}
\end{lem}
\begin{proof}
The first claim follows from exactness of coinduction. (In particular, the surjection $\wh{\bv}(\lambda) \twoheadrightarrow \wh{\irr}(\lambda)$ factors
through $\wh{\bm}_I(\lambda)$.)  The second claim can be deduced from \eqref{eqn:parabolicT-induction-tensor}. 
\end{proof}
\begin{rmk}
An immediate consequence of the first statement in this lemma is that 
\[
\wh{\bm}_I(\lambda)/\rad^1\,\wh{\bm}_I(\lambda) \cong \wh{\irr}(\lambda).
\]
In particular, $\wh{\bm}_I(\lambda)$ has an irreducible head, and hence, is indecomposable. 
\end{rmk}

We now prove an analogue to \cite[Lemma II.2.11]{jantzen}. 
\begin{lem}\label{lem:U-invariants}
Let $\lambda \in \bX$ be arbitrary and regard $\wh{\irr}(\lambda)$ as a $P_1^+T$-module, then 
there is an $L_1T$-module isomorphism
\[
\wh{\irr}(\lambda)^{(U_I^+)_1} \cong \wh{\irr}_I(\lambda).
\]
\end{lem}
\begin{proof}
Let $\wh{\bm}'_I(\lambda) = \Ind_{P_1T}^{G_1T}\, \wh{\irr}_I(\lambda)$ as in Remark~\ref{rmk:dual-para-baby}. The ``dual" of \eqref{eqn:parabolicT-induction-tensor} gives
a $P_1^+T$-module isomorphism
\[
\wh{\bm}'_I(\lambda) \cong \bk[(U_I^+)_1] \otimes \wh{\irr}_I(\lambda),
\]
where, in particular, $(U_I^+)_1$ acts via the left regular representation on the first term and trivially on the second term. Thus, 
\[
\wh{\bm}'_I(\lambda)^{(U_I^+)_1} \cong \bk[(U_I^+)_1]^{(U_I^+)_1} \otimes \wh{\irr}_I(\lambda) \cong \wh{\irr}_I(\lambda),
\]
where the last isomorphism follows from the fact that $\bk[(U_I^+)_1]^{(U_I^+)_1} = \bk$.

Moreover, the ``dual" of Lemma~\ref{lem:parabolic-vermas} gives an inclusion $\wh{\bm}'_I(\lambda) \hookrightarrow \wh{\bv}'(\lambda)$.
Thus, there also exists an inclusion $\wh{\irr}(\lambda) \hookrightarrow \wh{\bm}'_I(\lambda)$ since 
$\wh{\bm}'_I(\lambda)$ must contain the (simple) socle of $\wh{\bv}'(\lambda)$. This implies 
\[
\wh{\irr}(\lambda)^{(U_I^+)_1} \subseteq \wh{\bm}'_I(\lambda)^{(U_I^+)_1} \cong \wh{\irr}_I(\lambda).
\]
Finally, since the first term is non-zero (because non-zero $(U_I^+)_1$ invariants always exist), and the middle term is irreducible for $L_1T$, then 
we must have equality. 
\end{proof}

The next proposition will be essential to our Loewy series calculations.
\begin{prop}\label{prop:coinduced-head}
Let $M$ be an arbitrary $L_1T$-module (regarded as a $P_1^+T$-module with a trivial $(U_I^+)_1$ action),
and let $N = \coInd_{P_1^+T}^{G_1T}\, M$, then 
\[
[\radl_0\,N:\wh{\irr}(\mu)] = [\radl_0 \,M:\wh{\irr}_I(\mu)]
\] 
for all $\mu \in \bX$.
\end{prop}
\begin{proof}
It suffices to show that $\dim \Hom_{G_1T}(N, \wh{\irr}(\mu)) =  \dim\Hom_{L_1T}(M, \wh{\irr}_I(\mu))$ for all $\mu \in \bX$. 
Observe,
\[
\Hom_{G_1T}(N, \wh{\irr}(\mu)) \cong \Hom_{P_1^+T}(M, \wh{\irr}(\mu)) \cong \Hom_{L_1T}\left(M, \wh{\irr}(\mu)^{(U_I^+)_1}\right),
\]
 where the first isomorphism follows from \cite[I.8.14(4)]{jantzen} and the second isomorphism holds because $(U_I^+)_1$ acts trivially on 
 $M$, so the image of any morphism must also be $(U_I^+)_1$-invariant. Finally, by Lemma~\ref{lem:U-invariants}, 
 \[
\Hom_{L_1T}\left(M, \wh{\irr}(\mu)^{(U_I^+)_1}\right) \cong \Hom_{L_1T}(M, \wh{\irr}_I(\mu)).
 \]
Therefore, $\dim \Hom_{G_1T}(N, \wh{\irr}(\mu)) =  \dim\Hom_{L_1T}(M, \wh{\irr}_I(\mu))$. 
\end{proof}

\section{Multiplicity and dimension formulas}\label{sec:initial-results}
We now fix two subsets $I = \{\epsilon_1-\epsilon_2, \dots, \epsilon_{n-1}-\epsilon_n\}$ and 
 $J = \{\epsilon_2-\epsilon_3, \dots, \epsilon_{n}-\epsilon_{n+1}\}$ of $S$ for the rest of this paper. 
Our first result gives a remarkable property of $\cC(\lambda_0)$. 
\begin{prop}\label{prop:weyl-irreducible}
For $i=0,\dots, n$,
\[
\Weyl(\lambda_i)|_{G_1} \cong \irr(\lambda_i),
\]
where $\Weyl(\lambda)$ denotes the Weyl module of highest weight $\lambda \in \bX$.
\end{prop}
\begin{proof}
We will use Jantzen's criterion for the simplicity of Weyl modules (cf. \cite[II.8.21]{jantzen} or \cite{jantzen-weyl}). 
Set $\nu_i =\lambda_i + \rho$ for $i=0,\dots,n$. It will suffice to show that for any $i = 0, \dots, n$ and $1 \leq k< j \leq n+1$,
the quantity
\begin{equation*}
\langle \nu_i, \epsilon_k - \epsilon_j \rangle
\end{equation*}
satisfies the criterion. 
We will proceed by dividing this problem into the following cases:
\begin{itemize}
\item[{\bf Case 1}] $i=0$
  \begin{itemize}
  \item [{\bf 1.1)}] $k = 1$, $2 \leq j \leq n+1$
  \item [{\bf 1.2)}]  $2 \leq k < j \leq n+1$ 
  \end{itemize}
\item[{\bf Case 2}] $i=n$
  \begin{itemize}
  \item [{\bf 2.1)}] $1\leq k \leq n$, $j=n+1$
  \item [{\bf 2.2)}]  $1 \leq k < j \leq n$
  \end{itemize}
  \item[{\bf Case 3}] $1 \leq i \leq n-1$
  \begin{itemize}
  \item [{\bf 3.1)}] $1 \leq k < j \leq i $
  \item [{\bf 3.2)}]  $i+2 \leq k < j \leq n+1$
  \item [{\bf 3.3)}] $1 \leq k \leq i$, $j=i+1$
   \item [{\bf 3.4)}] $k = i+1$, $i+1 < j \leq n+1$
    \item [{\bf 3.5)}] $1 \leq k \leq i < i+2 \leq j \leq n+1$.
  \end{itemize}
\end{itemize}
  (Note that some of these sub-cases may be empty for certain choices of $i$ and $n$.)

 Let us first consider {\bf Case 1}. We calculate
${\la \nu_0, \epsilon_1 - \epsilon_{2} \ra = 1}$, 
and for  $2\leq k\leq n$,  ${\la \nu_0, \epsilon_k - \epsilon_{k+1} \ra = p}$. 
In the situation of {\bf 1.1)}, first observe
\[
{\la \nu_0, \epsilon_1 - \epsilon_{j} \ra = 1 + (j-2)p},
\]
for $2 \leq j \leq n+1$. Following the notation in \cite[II.8.21]{jantzen}, set $a = 1$, $b = j-2$ and $s=0$. The criterion is satisfied by setting 
$\beta_0 = \epsilon_1 - \epsilon_{2}$ and $\beta_r = \epsilon_{r+1}-\epsilon_{r+2}$ for $r = 1,\dots, j-2$. 
Similarly, in the case of {\bf 1.2)}, first note that for
$2 \leq k < j \leq n+1$, 
\[
\la \nu_0, \epsilon_k - \epsilon_{j} \ra = (j-k)p.
\]
If we write ${j-k= ap^{s-1} + bp^s}$ for some $s \geq 1$ and $0 < a < p$, then
 ${\la \nu_0, \epsilon_k - \epsilon_{j} \ra} = ap^s + bp^{s+1}$. Finally, set 
 ${\beta_0 = \epsilon_k - \epsilon_{k+ ap^{s-1}}}$, and ${\beta_r = \epsilon_{k+ ap^{s-1}+(r-1)p^s} - \epsilon_{k+ ap^{s-1} + rp^s}}$ for 
 ${r= 1,\dots, b}.$
 
Now we consider {\bf Case 2}. Again, we calculate for $1\leq k\leq n-1$, 
$
 \la \nu_n, \epsilon_k - \epsilon_{k+1} \ra = p
$
and ${\la \nu_n, \epsilon_n - \epsilon_{n+1}\ra =p-1}$. In the situation of {\bf 2.1)}, note that for $1\leq k \leq n$, 
\[
{\la \nu_n, \epsilon_k - \epsilon_{n+1} \ra = (p-1) + (n-k)p},
\]
so $a = p-1$, $b = n-k$ and $s=0$. The criterion is satisfied by setting 
$\beta_0 = \epsilon_n - \epsilon_{n+1}$ and $\beta_r = \epsilon_{r+k-1}-\epsilon_{r+k}$ for $r = 1,\dots, n-k$.
In the situation of {\bf 2.2)}, we get 
\[
\la \nu_n, \epsilon_k - \epsilon_{j} \ra = (j-k)p,
\]
for $1 \leq k < j \leq n$. If we write ${j-k= ap^{s-1} + bp^s}$ for some $s \geq 1$ and $0 < a < p$, then
 ${\la \nu_n, \epsilon_k - \epsilon_{j} \ra} = ap^s + bp^{s+1}$. Now set
 ${\beta_0 = \epsilon_k - \epsilon_{ap^{s-1} +k}}$, and 
 ${\beta_r = \epsilon_{k+ ap^{s-1}+(r-1)p^s} - \epsilon_{k+ ap^{s-1} + rp^s}}$ for 
 ${r= 1,\dots, b}$ as in {\bf 1.2)}.

Finally we consider {\bf Case 3}. 
For  $1\leq k\leq i-1$ and $i+2 \leq k \leq n+1$, we have
$
 \la \nu_i, \epsilon_k - \epsilon_{k+1} \ra = p,
$
${ \la \nu_i, \epsilon_i - \epsilon_{i+1} \ra = p-1}$ and ${ \la \nu_i, \epsilon_{i+1} - \epsilon_{i+2} \ra = 1}$.
To handle {\bf 3.1)} and {\bf 3.2)}, note that in both instances
\[
\la \nu_i, \epsilon_k - \epsilon_{j} \ra = (j-k)p.
\]
These cases then follow by writing ${j-k= ap^{s-1} + bp^s}$ with $0 < a < p$,  and defining $\beta_0$ and $\beta_r$ for $r=1,\dots, b$ as in {\bf 1.2)} and {\bf 2.2)} respectively. 
Similarly, the verifications of {\bf 3.3)} and {\bf 3.4)} are identical to the verifications of {\bf 2.1)} and {\bf 1.1)} respectively. 

Thus, we only need to consider {\bf 3.5)}. We first observe
\[
\la \nu_i, \epsilon_k - \epsilon_{j} \ra = (j-k-1)p,
\]
and write ${j-k-1= ap^{s-1} + bp^s}$ where $0 < a < p$.  Now there are two further sub-cases:
\begin{itemize}
\item [{\bf a)}] $k + ap^{s-1} \geq i+1$,
\item [{\bf b)}] $k + ap^{s-1} \leq i$.
\end{itemize}
In  the situation of {\bf a)}, we set $\beta_0 = \epsilon_k - \epsilon_{k+1+ap^{s-1}}$ and  
${\beta_r = \epsilon_{k+ 1+ ap^{s-1}+(r-1)p^s} - \epsilon_{k+ 1+ ap^{s-1} + rp^s}}$ for 
 ${r= 1,\dots, b}$. In the situation of {\bf b)}, we set 
 ${\beta_0 =  \epsilon_k - \epsilon_{k+ap^{s-1}}}$ and 
 \[
 \beta_r = \begin{cases}
\epsilon_{k+ ap^{s-1}+(r-1)p^s} - \epsilon_{k+ ap^{s-1} + rp^s} & \text{if $k+ ap^{s-1} + rp^s \leq i$}\\
\epsilon_{k+ ap^{s-1}+(r-1)p^s} - \epsilon_{k+ 1+ap^{s-1} + rp^s} & \text{if $k+ ap^{s-1} + (r-1)p^s \leq i$ and} \\
 & \text{ $k+ ap^{s-1} + rp^s \geq i+1$}\\
\epsilon_{k+ 1+ ap^{s-1}+(r-1)p^s} - \epsilon_{k+ 1+ ap^{s-1} + rp^s} & \text{otherwise}\\
 \end{cases}
 \]
 for $r = 1,\dots, b$. 
 \end{proof}
 \begin{rmk}
If we assume $p > n+1$, then the preceding proof can be dramatically simplified since 
$\langle \nu_i, \epsilon_k - \epsilon_j \rangle < p^2$ for all possible choices of $i$, $j$, $k$. 
 \end{rmk}

\begin{cor}\label{cor:weyl-char}
The dimensions and characters of the irreducible modules in $\cC(\lambda_0)$ are given by Weyl's dimension formula and Weyl's character formula respectively. 
\end{cor}

Observe now that
\[
\Phi^+ \backslash \Phi_I^+ = \{ \epsilon_1 - \epsilon_{n+1}, \epsilon_{2}-\epsilon_{n+1},\dots, \epsilon_{n} - \epsilon_{n+1} \}
\]
which implies $\dim_{\bk}\pdist_I= p^n$, and that the lowest weight of $\pdist_I$ is
\begin{equation}\label{eqn:low-weight}
  \begin{aligned}
     \mu_I &= -(p-1)(\epsilon_1 - \epsilon_{n+1} + \epsilon_2 - \epsilon_{n+1} + \cdots + \epsilon_n - \epsilon_{n+1}) \\
               &= -(p-1)(n+1)\varpi_n.
  \end{aligned}
\end{equation}
The longest element $w_I \in W_I$  is the permutation given by
\begin{equation}\label{eqn:long-elementI}
w_I: i \mapsto \begin{cases}
   n+1 -i &\text{for $1 \leq i \leq n$},\\
   n+1 & \text{for $i=n+1$}.
   \end{cases}
\end{equation}
Analogously, the lowest weight of
$\pdist_J$ is 
\[
\mu_J = -(p-1)(n+1)\varpi_1
\]
and $w_J \in W_J$ is given by
\begin{equation}\label{eqn:long-elementJ}
w_J:  i \mapsto \begin{cases}
    1 & \text{for $i=1$},\\
   n+3 -i &\text{for $2 \leq i \leq n+1$}.
   \end{cases}
\end{equation}

Our goal is to explicitly describe the modules $\widehat{\bm}_I(\lambda_i)$ and $\widehat{\bm}_J(\lambda_i)$.
We begin with the following dimension formula for the restrictions of these modules to $G_1$.
\begin{lem}\label{lem:sing-dimension}
\quad
\begin{enumerate}[leftmargin=*]
\item For $i=0,\dots, n-1$, 
\[
  \dim_{\bk}\bm_I(\lambda_i) = \dim_{\bk} \irr(\lambda_i) + \dim_{\bk} \irr(\lambda_{i+1}).
\]
\item For $i=1,\dots, n$,
\[
  \dim_{\bk}\bm_J(\lambda_i) = \dim_{\bk} \irr(\lambda_i) + \dim_{\bk} \irr(\lambda_{i-1}).
\]
\end{enumerate}

\end{lem}
\begin{proof}
It will be enough to prove (1), since (2) will follow from the exact same arguments. 
For notational simplicity, set $\nu_i = \lambda_i + \rho$ for $i=0,\dots, n-1$. 
By Corollary~\ref{cor:weyl-char}, we can apply the Weyl dimension formula, which gives
\[
\dim_{\bk} \irr(\lambda_i) = \frac{\prod_{1 \leq k < j \leq n+1} \la \nu_i, \epsilon_k - \epsilon_j\ra }{\prod_{1 \leq k < j \leq n+1} \la \rho, \epsilon_k - \epsilon_j\ra},
\]
where
\[
\prod_{1 \leq k < j \leq n+1} \la \rho, \epsilon_k - \epsilon_j\ra = n!(n-1)!\cdots 2!1!.
\]
By \eqref{eqn:parabolic-induction-tensor}, the description of the weight basis of $\pdist_I$ in 
\eqref{eqn:weight-basis-dist}, and the analogue of Corollary~\ref{cor:weyl-char} for the Levi factor $L_I$ (applied to $\irr_I(\lambda_i)$), 
we can deduce
\[
\dim_{\bk}\bm_I(\lambda_i) = \frac{p^n\prod_{1 \leq k < j \leq n}  \la \nu_i, \epsilon_k - \epsilon_j\ra}{(n-1)!\cdots 2!1!}.
\]

Thus, the equation in the statement of the lemma is equivalent to 
\begin{equation}\label{eqn:numerator-sum}
\prod_{1 \leq k < j \leq n+1} \la \nu_i, \epsilon_k - \epsilon_j\ra + \prod_{1 \leq k < j \leq n+1} \la \nu_{i+1}, \epsilon_k - \epsilon_j\ra
 = n!p^n\prod_{1 \leq k < j \leq n}  \la \nu_i, \epsilon_k - \epsilon_j\ra,
\end{equation}
for $i=0,\dots, n-1$. 
Notice that if we treat $p$ as an indeterminate variable, then the expressions
\[
\la \nu_i, \epsilon_{i+1}-\epsilon_{n+1}\ra = (n-1-i)p + 1, \quad \la \nu_{i+1}, \epsilon_{1} - \epsilon_{i+2} \ra = (i+1)p - 1
\]
are exclusive to the first and second terms appearing in \eqref{eqn:numerator-sum} respectively. 
So it will be helpful to introduce the notation
\[
\Gamma = \{ (k,j) \, \mid \, 1 \leq k < j \leq n+1\}.
\] 
{\,\,\bf Claim.} 
For $i=0,\dots, n-1$, 
\[
\prod_{\Gamma\backslash \{(i+1,n+1)\}} \la \nu_i, \epsilon_k-\epsilon_j\ra = 
\prod_{\Gamma\backslash \{(1,i+2)\}}\la \nu_{i+1}, \epsilon_k-\epsilon_j\ra.
\]

Suppose for now that this claim holds, then by the observation immediately preceding the claim,
\[
\begin{aligned}
\prod_{1 \leq k < j \leq n+1} \la \nu_i, \epsilon_k - \epsilon_j\ra + \prod_{1 \leq k < j \leq n+1} \la \nu_{i+1}, \epsilon_k - \epsilon_j\ra &= 
 np \prod_{\Gamma\backslash \{(i+1,n+1)\}} \la \nu_i, \epsilon_k-\epsilon_j\ra.
\end{aligned}
\]
Combining this with the following identity:
\[
\begin{aligned}
\prod_{\Gamma\backslash \{(i+1,n+1)\}} \la \nu_i, \epsilon_k-\epsilon_j\ra 
    &=   \prod_{1\leq k \leq n, \, k\neq i+1} \la \nu_i, \epsilon_k-\epsilon_{n+1} \ra \prod_{1 \leq k < j \leq n}  \la \nu_i, \epsilon_k - \epsilon_j\ra \\
    &= (n-1)!p^{n-1} \prod_{1 \leq k < j \leq n}  \la \nu_i, \epsilon_k - \epsilon_j\ra,
 \end{aligned}
\]
verifies  \eqref{eqn:numerator-sum}.

The remainder of the proof will be devoted to proving the preceding claim. Let 
\[
X = \{(k,j) \in \Gamma \, \mid \, j = i+1,i+2, \text{ or } k = i+1, i+2\}.
\] 
It can be checked that $|X| = 2n-1$ and that $X$
is the subset of $\Gamma$ consisting of all $(k,j)$ satisfying
\[
\la \nu_i, \epsilon_k - \epsilon_j \ra \neq \la \nu_{i+1}, \epsilon_k - \epsilon_j \ra.
\] 
In particular, the sets ${X\backslash \{(i+1,n+1)\}}$ and ${X \backslash \{(1,i+2)\} }$ have precisely $2n-2$ elements. 
We get immediately that 
\[
\prod_{\Gamma\backslash X} \la \nu_i, \epsilon_k-\epsilon_j\ra = \prod_{\Gamma\backslash X}\la \nu_{i+1}, \epsilon_k-\epsilon_j\ra,
\]
and we only have to check the $(2n-2)$-fold products
\[
\prod_{X \backslash \{(i+1,n+1)\}} \la \nu_i, \epsilon_k-\epsilon_j\ra, \,\, 
\prod_{X \backslash \{(1,i+2)\}}\la \nu_{i+1}, \epsilon_k-\epsilon_j\ra.
\]
If $i=0$, 
\[
\begin{aligned}
\prod_{X \backslash \{(1,n+1)\}}& \la \nu_0, \epsilon_k-\epsilon_j\ra \\
&=  \bigg( \big(p\big)\big(2p\big)\cdots \big((n-1)p\big)\bigg) 
\bigg(\big(p+1\big)\big(2p+1\big)\cdots \big((n-2)p+1\big)\bigg)\\
			&= \prod_{X \backslash \{(1,2)\}} \la \nu_1, \epsilon_k-\epsilon_j\ra,
\end{aligned}
\]
and for $i=n-1$,
\[
\begin{aligned}
\prod_{X \backslash \{(n,n+1)\}}& \la \nu_{n-1}, \epsilon_k-\epsilon_j\ra \\
=&  \bigg(\big((n-1)p-1\big)\big((n-2)p-1\big)\cdots \big(p-1\big)\bigg) \\
&\times \bigg( \big((n-1)p\big)\big((n-2)p\big)\cdots \big(p\big)\bigg) \\
			=& \prod_{X \backslash \{(1,n+1)\}} \la \nu_n, \epsilon_k-\epsilon_j\ra.
\end{aligned}
\]
Finally, suppose that $1 \leq i \leq n-2$, then
\[
\begin{aligned}
\prod_{X \backslash \{(i+1,n+1)\}}& \la \nu_i, \epsilon_k-\epsilon_j\ra  \\
=&\bigg(\big(ip-1\big)\big((i-1)p-1\big)\cdots \big(p-1\big)\bigg)
\bigg( \big(ip\big)\big((i-1)p\big)\cdots \big(p\big)\bigg) \\
        &\times \bigg( \big(p+1\big)\big(2p+1\big)\cdots \big((n-2-i)p+1\big)\bigg) \\
        &\quad \times \bigg(\big(p\big)\big(2p\big)\cdots \big( (n-1-i)p\big)\bigg) \\
        =& \prod_{X \backslash \{(1,i+2)\}}\la \nu_{i+1}, \epsilon_k-\epsilon_j\ra.
\end{aligned}
\]
\end{proof}

 It will now be helpful to recall that if we let $\lambda \in \bX$ be arbitrary, and write $\lambda = \mu + p\nu$ for $\mu \in \bX_1$, $\nu \in \bX$, then the lowest weight of $\wh{\irr}(\lambda)$ is unique to $\lambda$ and is given by
 $
 w_0(\mu) + p\nu.
 $

\begin{lem}\label{lem:g1t-M_I-structure}
\quad
\begin{enumerate}[leftmargin=*]
\item For $i=0,\dots, n-1$, 
${
\widehat{\irr}(\lambda_{i+1} - p\varpi_n) \subset \widehat{\bm}_I(\lambda_i).
}$
\item For $i=1,\dots, n$, 
${
\widehat{\irr}(\lambda_{i-1} - p\varpi_1) \subset \widehat{\bm}_J(\lambda_i).
}$
\end{enumerate}
\end{lem}
\begin{proof}
The proofs of (1) and (2) are identical, so we will only prove (1). 
It will be sufficient to show that the lowest weights of $\widehat{\bm}_I(\lambda_i)$ and 
${\widehat{\irr}(\lambda_{i+1}-p\varpi_n)}$ coincide. This is because Lemma~\ref{lem:sing-dimension} will then imply that the modules
 $\widehat{\irr}(\lambda_{i+1} - p\varpi_n)$ and $\widehat{\irr}(\lambda_{i})$ account for the complete set of composition factors (including multiplicity) of $\widehat{\bm}_I(\lambda_i)$.
The fact that $\widehat{\irr}(\lambda_{i})$ is the unique simple quotient of $\widehat{\bm}_I(\lambda_i)$ by Lemma~\ref{lem:parabolic-vermas} will then force $\widehat{\irr}(\lambda_{i+1} - p\varpi_n)$ to be a submodule.

By Lemma~\ref{lem:parabolic-vermas} and \eqref{eqn:low-weight}, the lowest weight of $\widehat{\bm}_I(\lambda_{i})$ is 
\[
-(p-1)(n+1)\varpi_n + w_I(\lambda_{i}).
\]
Likewise, the lowest weight of $\widehat{\irr}(\lambda_{i+1}-p\varpi_n)$ is 
\[
w_0(\lambda_{i+1}) - p\varpi_n.
\]
Hence, the result will follow if we can prove that 
\begin{equation}\label{eqn:omega_n}
-(p-1)(n+1)\varpi_n + w_I(\lambda_{i}) - w_0(\lambda_{i+1}) = -p\varpi_n.
\end{equation}
(Recall the definitions \eqref{eqn:long-element}, \eqref{eqn:long-elementI}, and \eqref{eqn:long-elementJ}.)

To verify this identity, let us first define 
\[
\rho_I = \frac{1}{2}\sum_{\alpha \in \Phi_I^+}\alpha \in \frac{1}{2}\bX,
\]
then 
\[
\begin{aligned}
\rho &= \rho_I + \frac{1}{2}(\epsilon_1 - \epsilon_{n+1} + \epsilon_2 - \epsilon_{n+1} + \cdots + \epsilon_n - \epsilon_{n+1}) \\
    &= \rho_I + \frac{(n+1)}{2}\varpi_n.
\end{aligned}
\]
Thus, $w_0(\rho) = -\rho_I - \frac{(n+1)}{2}\varpi_n$ and $w_I(\rho)= -\rho_I + \frac{(n+1)}{2}\varpi_n.$
We also observe 
that for  $i=1,\dots, n-1$,   ${w_I(\varpi_{i}) = \varpi_n - \varpi_{n-1-i}}$ and, 
by recalling the $\mu_i$ from \eqref{eqn:mu-formula}, 
\[
\begin{aligned}
w_I(\mu_i) &= \epsilon_{n-i} + \rho_I - \frac {(n+1)}{2}\varpi_n, \\
w_0(\mu_{i+1}) &= \epsilon_{n-i} + \rho_I + \frac{(n+1)}{2}\varpi_n,
\end{aligned}
\]
for $i=0,\dots, n-1$.
Finally, we verify \eqref{eqn:omega_n} by substituting the preceding identities into \eqref{eqn:weight-formula2}.
\end{proof}


The following proposition will give us some insight into the structure of the modules 
$\wh{\bm}_I(\lambda_i + p\nu)$ and $\wh{\bm}_J(\lambda_i + p\nu)$.
\begin{prop}\label{prop:parabolic-coinduction1}
Let $\nu \in \bX$. 
\begin{enumerate}[leftmargin=*]
\item For $i=0,\dots,n-1$, $\widehat{\bm}_I(\lambda_i+p\nu)$ is an  indecomposable length 2 module where we have 
$\radl_0\,\widehat{\bm}_I(\lambda_i + p\nu) \cong  \widehat{\irr}(\lambda_i + p\nu)$ and $\radl_1\, \widehat{\bm}_I(\lambda_i+p\nu) \cong \widehat{\irr}(\lambda_{i+1} +p\nu - p\varpi_n)$. Also, 
\[
\widehat{\bm}_I(\lambda_n + p\nu) \cong \widehat{\bv}(\lambda_n + p\nu).
\]
\item For $i=1,\dots,n$, $\widehat{\bm}_J(\lambda_i + p\nu)$ is an  indecomposable length 2 module where we have
$\radl_0\, \widehat{\bm}_J(\lambda_i+p\nu) \cong \widehat{\irr}(\lambda_i+p\nu)$ and $\radl_1\, \widehat{\bm}_J(\lambda_i+p\nu) \cong \widehat{\irr}(\lambda_{i-1} + p\nu - p\varpi_1)$. Also,
\[
\widehat{\bm}_J(\lambda_0 + p\nu) \cong \widehat{\bv}(\lambda_0 +p\nu). 
\]
\item Restricting to $G_1$ gives  similar descriptions for $\bm_I(\lambda_i)$ and $\bm_J(\lambda_i)$. 
\end{enumerate}
\end{prop}
\begin{proof}
Without loss of generality, we can assume $\nu=0$. Also, 
it will be enough to prove (1), since (2) will follow from an identical argument.
 
 The description of $\widehat{\bm}_I(\lambda_i)$ for $i=0,\dots, n-1$ follows immediately from Lemmas~\ref{lem:sing-dimension} and \ref{lem:g1t-M_I-structure}. 
 On the other hand, the description of $\widehat{\bm}_I(\lambda_n)$ is a consequence of 
 \cite[Lemma II.11.8]{jantzen}.
 Namely, since ${\langle \lambda_n + \rho, \alpha^{\vee}\rangle = p}$ for all $\alpha \in I$, then it follows that
$
\widehat{\irr}_I(\lambda_n) \cong \widehat{\bv}_{I}(\lambda_n).
$
By transitivity of coinduction, we obtain
 \[
 \widehat{\bm}_I(\lambda_n) \cong \coInd_{(P_I^+)_1T}^{G_1T}\coInd_{B_1^+T}^{(P_I^+)_1T}\lambda_n \cong \widehat{\bv}(\lambda_n).
 \]
\end{proof}

We can now compute the
composition multiplicities of the baby Verma modules and the indecomposable
projective modules of $\cC(\lambda_0)$. Surprisingly, the formulas in this case will not depend on $i$. 

\begin{prop}\label{prop:bv_multiplicity}
For $0\leq i,j \leq n$, 
\[
[\bv(\lambda_i): \irr(\lambda_j)] = {{n}\choose{j}}. 
\]
\end{prop}
\begin{proof}
We shall perform induction with respect to $n$, where $G = SL_{n+1}$. For the base case, when $n=1$, we have 
\[
\lambda_0 = 0, \quad \lambda_1 = p-2.
\] 
In this case, $\cC(\lambda_0)$ is actually the regular block and the claim can be verified through explicit computation (eg.
\cite[II.9.10]{jantzen}). 
Now suppose $n \geq 2$ and that the formula holds for $SL_{r+1}$ with $r \leq n-1$.
We set $G' = [L_I, L_I]$ and note that $G' \cong SL_{n}$. 
Similarly, set $B' = B \cap G' \subset B \cap L_I$ and $T'  = T \cap G'$, where $B'$ is the (lower) Borel subgroup 
of $G'$ and $T'$ is the torus of $G'$.

By the remark following  \cite[Proposition I.8.20]{jantzen}, 
\[
\bv_I(\lambda)|_{G_1'} \cong \coInd_{(B'^+)_1}^{G'_1}(\lambda|_{T'})
\] 
and
by \cite[II.2.10(2)]{jantzen} 
\[
\irr_I(\lambda)|_{G'_1} \cong \irr(\lambda|_{T'}).
\] 
For $i=0,\dots, n-1$, set $\lambda'_i = \lambda_i|_{T'}$. Applying the inductive hypothesis to $G'$ yields
\[
[\bv_I(\lambda_i)|_{G_1'}: \irr_I(\lambda_j)|_{G_1'}] = [\coInd_{(B'^+)_1}^{G'_1}(\lambda'_i): \irr(\lambda'_j)]  = {n-1\choose j}
\]
for $0 \leq i,j \leq n-1$. The modules $\irr_I(\lambda_0),\dots, \irr_I(\lambda_{n-1})$ form the entire set of irreducibles of the $\Rep((L_I)_1)$ block
 $\cC_I(\lambda_0)$. Now since 
$\bv_I(\lambda_i)$ is an object of  $\cC_I(\lambda_0)$ and for each $i=0,\dots, n-1$, $\irr_I(\lambda_i)$ is the only irreducible of
$\cC_I(\lambda_0)$ which satisfies $\irr_I(\lambda_i)|_{G_1'}\cong \irr(\lambda'_i)$, then we must also have 
\begin{equation*}
[\bv_I(\lambda_i): \irr_I(\lambda_j)] = {n-1\choose j}.  
\end{equation*}

If we take any Jordan-H\"older filtration of $\bv_I(\lambda_i)$ for $0 \leq i \leq n-1$ and apply the exact functor $\coInd_{(P_I)^+_1}^{G_1}(-)$, 
we will get a filtration whose layers are of the form $\bm_I(\lambda_j)$ for $0 \leq j \leq n-1$. Thus, 
\begin{equation}\label{eqn:inductive-formula}
[\bv(\lambda_i): \bm_I(\lambda_j)] = [\bv_I(\lambda_i): \irr_I(\lambda_j)] = {n-1\choose j},
\end{equation}
where $[\bv(\lambda_i): \bm_I(\lambda_j)]$ denotes the filtration multiplicity. 

By Proposition~\ref{prop:parabolic-coinduction1}, each $\bm_I(\lambda_j)$ contributes a single copy of 
$\irr(\lambda_j)$ and $\irr(\lambda_{j+1})$. Thus, $[\bv(\lambda_i): \irr(\lambda_0)] = 1$ 
and $[\bv(\lambda_i): \irr(\lambda_n)] = 1$ since by \eqref{eqn:inductive-formula}, $[\bv(\lambda_i): \bm_I(\lambda_0)] =1$ and 
$[\bv(\lambda_i): \bm_I(\lambda_{n-1})] =1$. Likewise, 
for $1 \leq j \leq n-1$ the multiplicity
\[
[\bv(\lambda_i): \irr(\lambda_j)] = {n-1\choose j} + {n-1\choose j-1}
\]
arises from the $ {n-1\choose j}$ copies of $\bm_I(\lambda_j)$ and the $ {n-1\choose j-1}$ copies of 
$\bm_I(\lambda_{j-1})$. The proposition now follows from the well-known identity
\[
{n\choose j} = {n-1 \choose j} + {n-1 \choose j-1}. 
\]

Thus, we have verified the formula for $\bv(\lambda_i)$ when $0 \leq i \leq n-1$. The $\bv(\lambda_n)$ case can 
be verified by replacing $I$ with $J$ and repeating the same arguments. 
\end{proof}
\begin{rmk}
An alternative argument is to simply apply Theorem~\ref{thm:g1t-radical-layers}, whose proof is independent of 
this proposition. 
\end{rmk}

If we let ${[\Inj(\lambda):\bv(\mu)]}$ denote the multiplicity of $\bv(\mu)$ in any baby Verma
filtration as in \cite[Proposition II.11.4]{jantzen}, then the following identity
\begin{equation*}\label{eqn:bgg-recip1}
[\Inj(\lambda):\bv(\mu)] = [\bv(\mu): \irr(\lambda)]
\end{equation*}
is known as \emph{BGG reciprocity} for baby Verma modules.
\begin{cor}
  For $0\leq i,j \leq n$, 
  \begin{equation*}
    [\Inj(\lambda_i):\irr(\lambda_j)]=(n+1){n \choose i} {n\choose j}. 
  \end{equation*}
\end{cor}
\begin{proof}
By BGG reciprocity,  
\[
\begin{aligned}
      {[\Inj(\lambda_i):\irr(\lambda_j)]} &= \sum_{k=0}^{n}[\Inj(\lambda_i):\bv(\lambda_k)][\bv(\lambda_k):\irr(\lambda_j)] \\
    				&= \sum_{k=0}^{n}[\bv(\lambda_k):\irr(\lambda_i)][\bv(\lambda_k):\irr(\lambda_j)] \\
				&= (n+1){n \choose i}{n \choose j},
\end{aligned}
\]
where the last equality follows from Proposition~\ref{prop:bv_multiplicity}. 
\end{proof}

\section{Extensions between irreducibles}\label{sec:ext-irred}
The goal of this section is to prove
the following theorem. 
\begin{thm}\label{thm:ext1-thm}
Let $V = \irr(\varpi_1)$ be the standard representation for $G$, then for $0 \leq i,j \leq n$
\[
\Ext_{G_1}^1(\irr(\lambda_i), \irr(\lambda_j))^{(-1)} = 
   \begin{cases}
   V & \text{if $(i,j) = (i+1,i)$}, \\ 
   V^* &\text{if $(i,j) = (i,i+1)$}, \\
   0 & \text{otherwise},
    \end{cases}
\]
where $(-)^{(-1)}$ denotes the inverse Frobenius twist (see \cite[I.9]{jantzen}). 
\end{thm}

It will be helpful to recall the identity 
\begin{equation}\label{eqn:ext-weight-space}
\Ext_{G_1T}^1(\wh{\irr}(\lambda), \wh{\irr}(\mu -p\nu)) \cong \bigg(\Ext_{G_1}^1(\irr(\lambda), \irr(\mu))\bigg)_{p\nu}
\end{equation}
for any $\lambda, \mu \in \bX_1$ and $\nu \in \bX$, where the right hand side denotes the $p\nu$-weight space of the corresponding $G$-module\footnote{It also has the structure of a $G/G_1$-module.} (cf. \cite[I.6.9(4), (5)]{jantzen} and \cite[II.9.19(3)]{jantzen}). 

We also recall from \S\ref{sec:introduction} that our assumption of $p$ being very good for $SL_{n+1}$ (i.e. $p \nmid n+1$) is equivalent to the condition that 
the quotient $\bX/ \Z\Phi$ contains no $p$-torsion. In particular, 
\begin{equation}\label{eqn:weight-root-intersect}
p\bX\cap \Z\Phi = p\Z\Phi. 
\end{equation}
As a consequence, for any $\mu, \nu \in \bX$ 
\begin{equation}\label{eqn:very-good-consequence}
p\nu \leq p\mu \iff \nu \leq \mu. 
\end{equation}

Before proceeding to the proof of the theorem, we will record the following corollary. 
\begin{cor}\label{lem:rad1-injective}
 Let $G=SL_{n+1}$ with $n\geq 1$, then for $0\leq i \leq n$,
\[
\begin{aligned}
\radl_1\, \widehat{\Inj}(\lambda_0) &= \bigoplus_{k=1}^n \widehat{\irr}(\lambda_1 -p\varpi_{n+1-k} + p\varpi_{n+2-k}), \\
\radl_1\, \widehat{\Inj}(\lambda_n) &= \bigoplus_{k=1}^n \widehat{\irr}(\lambda_{n-1} - p\varpi_{k+1} + p\varpi_k),
 \end{aligned}
\]
and for $i=1,\dots, n-1$,
\[
\begin{aligned}
\radl_1\, \widehat{\Inj}(\lambda_i) =& \left(\bigoplus_{k=1}^n\, \widehat{\irr}(\lambda_{i-1} -p\varpi_k + p\varpi_{k-1})\right)\\
   &\oplus \left(\bigoplus_{k=1}^{n}\, \widehat{\irr}(\lambda_{i+1} -p\varpi_{n+1-k} + p\varpi_{n+2-k})\right),
\end{aligned}
\]
where we set $\varpi_0 =0$ and $\varpi_{n+1} =0$ for notational simplicity. 
\end{cor}
\begin{proof}
First recall that for any $\lambda, \mu \in \bX$,
\begin{equation}\label{eqn:rad-ext}
[\radl_1\, \wh{\Inj}(\lambda):\wh{\irr}(\mu)] = \dim\Ext^1_{G_1T}(\wh{\irr}(\lambda), \wh{\irr}(\mu)),
\end{equation}
which we obtain by applying $\Hom_{G_1T}(-, \wh{\irr}(\mu))$ to the short exact sequence 
\[
0 \rightarrow \rad^1\, \wh{\Inj}(\lambda) \rightarrow \wh{\Inj}(\lambda) \rightarrow \wh{\irr}(\lambda) \rightarrow 0. 
\]
On the other hand, if we combine Theorem~\ref{thm:ext1-thm} with \eqref{eqn:ext-weight-space}, then 
we can deduce that the dimensions 
\[
\dim\Ext^1_{G_1T}(\wh{\irr}(\lambda_i), \wh{\irr}(\lambda_j -p\nu)),
\]
for $0 \leq i,j \leq n$ and $\nu \in \bX$ are given by the appropriate weight multiplicities of $V$ and $V^*$.  
\end{proof}


Determining the top two radical layers of the $\bv(\lambda_i)$ will also be essential to our $\Ext^1$-calculation. 
Before stating this result, we will introduce some additional notation. 

First fix $I, J \subset S$ as in \S\ref{sec:initial-results}, and for $0 \leq i \leq n-1$, set
\begin{equation}\label{eqn:I-filtration}
\nolF_I^j(\lambda_i) = \coInd_{(P_I^+)_1}^{G_1}(\rad^j\, \bv_I(\lambda_i)).
\end{equation}
Similarly, for $1 \leq i \leq n$, set 
\begin{equation}\label{eqn:J-filtration}
\nolF_J^j(\lambda_i) = \coInd_{(P_J^+)_1}^{G_1}(\rad^j\, \bv_J(\lambda_i)).
\end{equation}
We also set $\olF_I^j(\lambda_i) = \nolF_I^j(\lambda_i)/\nolF_I^{j+1}(\lambda_i)$ and 
$\olF_J^j(\lambda_i) = \nolF_J^j(\lambda_i)/\nolF_J^{j+1}(\lambda_i)$.

The exactness of coinduction implies 
\begin{equation}\label{eqn:coind-exact}
\olF_I^j(\lambda_i) = \frac{\coInd_{(P_I^+)_1}^{G_1}(\rad^j\, \bv_I(\lambda_i))}{\coInd_{(P_I^+)_1}^{G_1}(\rad^{j+1}\, \bv_I(\lambda_i))}
  \cong \coInd_{(P_I^+)_1}^{G_1}(\radl_j\, \bv_I(\lambda_i)),
\end{equation}
with a similar statement for $P_J$. 

\begin{lem}\label{lem:interior-rad1}
Set $\lambda_{-1} = \lambda_{n+1} = 0$ and declare that $M^{\oplus 0} = 0$ for any module $M$,
then
\[
\radl_1\, \bv(\lambda_i) = 
\irr(\lambda_{i-1})^{\oplus i} \oplus \irr(\lambda_{i+1})^{\oplus n-i}.
\]
 
\end{lem}
\begin{proof}
The case for $n=1$ follows from \cite[II.9.10]{jantzen} and the $n=2$ case follows from \cite[Theorems 2.4-2.5]{xi1999}. 
Now suppose $n > 2$ and that the statement of the lemma holds for $SL_{r+1}$ whenever $1 \leq r < n$. 
By the same argument as in 
the proof of Proposition~\ref{prop:bv_multiplicity}, we can assume the statement also holds for
$\bv_I(\lambda_i)$ with $0 \leq i \leq n-1$ (respectively for
 $\bv_J(\lambda_i)$ with $1 \leq i \leq n$).  

The inductive hypothesis gives
\begin{equation}\label{eqn:rad1-levi}
\radl_1\, \bv_I(\lambda_i) =
    \irr_I(\lambda_{i-1})^{\oplus i} \oplus \irr_I(\lambda_{i+1})^{\oplus n-1-i} 
\end{equation}
for $0 \leq i \leq n-1$,
and 
\[
\radl_1\, \bv_J(\lambda_i) =
    \irr_J(\lambda_{i-1})^{\oplus i-1} \oplus \irr_J(\lambda_{i+1})^{\oplus n-i}
\]
for $1 \leq i \leq n$. 
Now we coinduce to get
\[
\olF_I^0(\lambda_i) = \coInd_{(P_I^+)_1}^{G_1}(\radl_0\, \bv_I(\lambda_i)) = \bm_I(\lambda_i),
\]
 and
\begin{equation}\label{eqn:olF1-formula}
\olF_I^1(\lambda_i) = \coInd_{(P_I^+)_1}^{G_1}(\radl_1\, \bv_I(\lambda_i)) =
    \bm_I(\lambda_{i-1})^{\oplus i} \oplus \bm_I(\lambda_{i+1})^{\oplus n-1-i}
\end{equation}
for $0 \leq i \leq n-1$. (The formulas for $\olF_J^0(\lambda_i)$ and $\olF_J^1(\lambda_i)$ are similar.)

Let us focus on $I$ for now. By Proposition~\ref{prop:parabolic-coinduction1}, 
\begin{equation}\label{eqn:olF_O}
\radl_0\, \olF_I^0(\lambda_i) = \irr(\lambda_i), \quad \radl_1\, \olF_I^0(\lambda_i) = \irr(\lambda_{i+1}).
\end{equation}
We also have
\begin{equation}\label{eqn:rad0-olF1}
\radl_0\, \nolF_I^1(\lambda_i) =
    \irr(\lambda_{i-1})^{\oplus i} \oplus \irr(\lambda_{i+1})^{\oplus n-1-i} = \radl_0\, \olF_I^1(\lambda_i)
\end{equation}
for $0 \leq i \leq n-1$, where the first isomorphism is a consequence of \eqref{eqn:rad1-levi} and
 Proposition~\ref{prop:coinduced-head} and the second isomorphism is deduced from \eqref{eqn:olF1-formula}
 and Proposition~\ref{prop:parabolic-coinduction1}. It then follows that
 the surjective map 
 \[
 \radl_0\, \nolF_I^1(\lambda_i) \twoheadrightarrow \frac{\nolF_I^1(\lambda_i)}{\rad^1\, \nolF_I^1(\lambda_i) + \nolF_I^2(\lambda_i)} = \radl_0\, \olF_I^1(\lambda_i)
 \]
 is an isomorphism since the left and right hand sides have the same dimension. 
Consequently, 
 \begin{equation}\label{eqn:F2-rad-inclusion}
 \nolF_I^2(\lambda_i) \subseteq \rad^1 \nolF_I^1(\lambda_i).
 \end{equation}
 
 Now let $M = \rad^1\, \bv(\lambda_i) / \rad^1\nolF_I^1(\lambda_i)$ and observe
 \begin{multline*}
 [M] = [\rad^1\, \bv(\lambda_i)] - [\rad^1\nolF_I^1(\lambda_i)] =
  [\bv(\lambda_i)] - [\irr(\lambda_i)]-[\nolF_I^1(\lambda_i)] + 
 [\radl_0\, \nolF_I^1(\lambda_i)] \\
=[\olF_I^0(\lambda_i)] - [\irr(\lambda_i)] + [\radl_0\, \olF_I^1(\lambda_i)].
 \end{multline*}
 Thus, if we combine this with \eqref{eqn:olF_O} and \eqref{eqn:rad0-olF1}, we deduce the formula
 \[
 [M] = i[\irr(\lambda_{i-1})] + (n-i)[\irr(\lambda_{i+1})].
 \]

Now observe that $\rad^1\,\nolF_I^1(\lambda_i) \subseteq \rad^2\, \bv(\lambda_i)$ since 
$\nolF_I^1(\lambda_i) \subseteq \rad^1\, \bv(\lambda_i)$. 
As a consequence, 
\[
\rad^1\, M = \frac{\rad^2\, \bv(\lambda_i) + \rad^1\, \nolF_I^1(\lambda_i)}{\rad^1\,\nolF_I^1(\lambda_i)} = \frac{\rad^2\, \bv(\lambda_i)}{\rad^1\,\nolF_I^1(\lambda_i)},
\]
and hence, 
\[
\radl_0\, M = \frac{\rad^1\, \bv(\lambda_i)}{\rad^1\,\nolF_I^1(\lambda_i)} \bigg / \frac{\rad^2\, \bv(\lambda_i)}{\rad^1\,\nolF_I^1(\lambda_i)} = \radl_1\, \bv(\lambda_i). 
\]
Moreover, 
$\radl_0\, \nolF_I^1(\lambda_i) = \nolF_I^1(\lambda_i)/\rad^1\nolF_I^1(\lambda_i)
\hookrightarrow M.
$
Thus by  \eqref{eqn:rad0-olF1}, there exists an injective map
\[
\iota:  \irr(\lambda_{i-1})^{\oplus i} \oplus \irr(\lambda_{i+1})^{\oplus n-1-i} \hookrightarrow M,
\]
whose image must account for every irreducible factor of $M$ except for a single copy of 
$\irr(\lambda_{i+1})$. 
Consequently, $M$ fits into a short exact sequence of the form
\[
0 \longrightarrow \irr(\lambda_{i-1})^{\oplus i} \oplus \irr(\lambda_{i+1})^{\oplus n-1-i} \xrightarrow{\iota} M 
    \longrightarrow \irr(\lambda_{i+1}) \longrightarrow 0.
\]

Now let $N = \iota\big(\irr(\lambda_{i-1})^{\oplus i}\big)$ so that
$[M/N] = (n-i)[\irr(\lambda_{i+1})]$. From
\cite[Proposition II.12.9]{jantzen}, we know that any $G_1$-module with precisely one isotypic component must be 
semisimple. It follows that $M/N \cong \irr(\lambda_{i+1})^{\oplus n-i}$. In particular, we get a surjection
\[
\rad^1\, \bv(\lambda_i) \twoheadrightarrow M/N \cong \irr(\lambda_{i+1})^{\oplus n-i}.
\]
Finally, since every map from $\rad^1\, \bv(\lambda_i)$ to a semisimple module factors through
$\radl_1\, \bv(\lambda_i)$ (recall that the latter is the head of the former), we get a surjection
\begin{equation}\label{eqn:rad1-inclusion1}
\radl_1\, \bv(\lambda_i) \twoheadrightarrow \irr(\lambda_{i+1})^{\oplus n-i}. 
\end{equation}
(Recall that $\radl_0\, M = \radl_1\, \bv(\lambda_i)$.)
Observe now that if $i=0$, then the preceding map must be an isomorphism since every possible factor of 
$\radl_1\, \bv(\lambda_i)$ has been accounted for and we are done. 

On the other hand, if $1 \leq i \leq n$, then by replacing $I$ with $J$ and repeating the same arguments, we also obtain a surjection
\begin{equation}\label{eqn:rad1-inclusion2}
\radl_1\, \bv(\lambda_i) \twoheadrightarrow \irr(\lambda_{i-1})^{\oplus i}.
\end{equation}
Therefore, the lemma follows by combining \eqref{eqn:rad1-inclusion1} and \eqref{eqn:rad1-inclusion2} which 
account for every possible factor of $\radl_1\, \bv(\lambda_i)$. 
\end{proof}

 We can now compute the top two radical layers of $\widehat{\bv}(\lambda_i)$. 
\begin{lem}\label{prop:rad1G1T}
Let us set $\varpi_0=0$ and $\varpi_{n+1} =0$ for notational simplicity. We then have
\[
\begin{aligned}
\radl_1\,& \widehat{\bv}(\lambda_0) = \\
& \widehat{\irr}(\lambda_1 - p\varpi_n) \oplus \widehat{\irr}(\lambda_1 - p\varpi_{n-1}+p\varpi_n) \oplus \cdots \oplus \widehat{\irr}(\lambda_1 - p\varpi_1 + p\varpi_2), \\
\radl_1\,& \widehat{\bv}(\lambda_n) =\\
& \widehat{\irr}(\lambda_{n-1} - p\varpi_1) \oplus \widehat{\irr}(\lambda_{n-1} - p\varpi_{2}+p\varpi_1) \oplus \cdots \oplus \widehat{\irr}(\lambda_{n-1} - p\varpi_{n} + p\varpi_{n-1}),
 \end{aligned}
\]
and for $i=1,\dots, n-1$,
\[
\begin{aligned}
\radl_1\, \widehat{\bv}(\lambda_i) =& \left(\bigoplus_{k=1}^i\, \widehat{\irr}(\lambda_{i-1} -p\varpi_k + p\varpi_{k-1})\right) \\
     &\oplus 
   \left(\bigoplus_{k=1}^{n-i}\, \widehat{\irr}(\lambda_{i+1} -p\varpi_{n+1-k} + p\varpi_{n+2-k})\right).
   \end{aligned}
\]
\end{lem}
\begin{proof}
The $n=1$ case follows from \cite[II.9.10]{jantzen}, and the 
 $n=2$ case is given in \cite[Theorems 2.4-2.5]{xi1999}. Suppose now that $n>2$, and that the statement holds for
 $SL_{r+1}$ with $2 \leq r < n$. The inductive hypothesis can be applied to $L_I$ and $L_J$ as in the proof of 
 Proposition~\ref{prop:bv_multiplicity}. More precisely, for $L_I$ and $i=0,\dots, n-1$, 
  \cite[Lemma II.9.2(3)]{jantzen} implies that the highest weight of
 every composition factor of $\widehat{\bv}_I(\lambda_i)$ is of the form $\lambda_i - \gamma$ for various
 $\gamma \in \Z I$. 
 
 If we set $G' = [L_I, L_I]\cong SL_n$, and $T' = T\cap G'$,  we get
\[
\widehat{\bv}_I(\lambda_i)|_{G'_1T'} \cong \widehat{\bv}(\lambda_i|_{T'}), \quad 
	\widehat{\irr}(\lambda_i - \gamma)|_{G'_1T'} \cong \widehat{\irr}(\lambda_i|_{T'} - \gamma|_{T'}),
\]
with ${[\widehat{\bv}_I(\lambda_i): \widehat{\irr}(\lambda_i - \gamma)] = 
    [\widehat{\bv}(\lambda_i|_{T'}): \widehat{\irr}(\lambda_i|_{T'} - \gamma|_{T'})] }$. Furthermore, 
    \begin{equation}\label{eq:gamma-restrict}
    \gamma|_{T'} = \sum_{i =1}^{n-1} a_i (\epsilon_i - \epsilon_{i+1})|_{T'} \iff 
    \gamma = \sum_{i =1}^{n-1} a_i (\epsilon_i - \epsilon_{i+1}).
    \end{equation}
 So the inductive hypothesis is applied to $L_I$ by first expressing the irreducibles occurring in the inductive
  hypothesis for $SL_{n}$ as $\widehat{\irr}(\lambda_i|_{T'} - \gamma|_{T'})$ for various uniquely determined $\gamma$, and then employing 
  \eqref{eq:gamma-restrict} to obtain the corresponding formulas for $\radl_1\, \widehat{\bv}_I(\lambda_i)$. (The case for $L_J$ with
 $i=1,\dots, n$ is similar.)

Thus, the inductive hypothesis gives
\[
\radl_1\, \widehat{\bv}_I(\lambda_0) = \widehat{\irr}_I(\lambda_1 - p\varpi_{n-1} + p\varpi_{n}) \oplus \widehat{\irr}_I(\lambda_1 - p\varpi_{n-2}+p\varpi_{n-1}) \oplus \cdots \oplus \widehat{\irr}_I(\lambda_1 - p\varpi_1 + p\varpi_2).
\]
So the formula for $\radl_1\, \widehat{\bv}(\lambda_0)$ follows from Proposition~\ref{prop:parabolic-coinduction1} and the proof of Lemma~\ref{lem:interior-rad1}.
For $i=1,\dots, n-1$, the hypothesis also gives
\[
\begin{aligned}
\radl_1\, \widehat{\bv}_I(\lambda_i) =& \left(\bigoplus_{k=1}^{i}\, \widehat{\irr}_I(\lambda_{i-1} -p\varpi_k + p\varpi_{k-1})\right) \\
  &\oplus
   \left(\bigoplus_{k=1}^{n-1-i}\, \widehat{\irr}_I(\lambda_{i+1} -p\varpi_{n+1-k} + p\varpi_{n+2-k})\right).
 \end{aligned}
\]
Now once again, the formula for $\radl_1\, \widehat{\bv}(\lambda_i)$ is obtained by applying Proposition~\ref{prop:parabolic-coinduction1}
and then proceeding as in the proof of Lemma~\ref{lem:interior-rad1}. 

Finally, the formula for $\radl_1\, \widehat{\bv}(\lambda_n)$ is verified by first applying the inductive hypothesis to 
$L_J$, which gives
\begin{multline*}
\radl_1\, \widehat{\bv}_J(\lambda_n) = \\
  \widehat{\irr}_J(\lambda_{n-1} - p\varpi_2+p\varpi_1) \oplus \widehat{\irr}_J(\lambda_{n-1} - p\varpi_{3}+p\varpi_2) \oplus \cdots \oplus \widehat{\irr}_J(\lambda_{n-1} - p\varpi_{n} + p\varpi_{n-1}), 
\end{multline*}
and then proceeding as above.  
\end{proof}

The following lemma helps characterize modules which admit a surjective map from a baby Verma module.
\begin{lem}\label{lem:Baby-quotient}
If ${E \in \Ext^1_{G_1T}(\widehat{\irr}(\lambda), \widehat{\irr}(\mu))}$ for $\lambda,\mu, \nu \in \bX$ is non-trivial and 
$\lambda \not\leq \mu$, then $E$ is a quotient of $\widehat{\bv}(\lambda)$. 
\end{lem}
\begin{proof}
By definition, $E$ is an indecomposable length 2 module with head $\widehat{\irr}(\lambda)$ and socle 
$\widehat{\irr}(\mu)$. In particular, $E$ is a cyclic module for $\uenvt$ which is generated by some 
$\lambda$-weight vector $v_{\lambda}$. 
Now every weight $\gamma$ occurring with non-zero multiplicity in $\widehat{\irr}(\mu)$ satisfies $\gamma \leq \mu$, and hence, 
$\lambda \not\leq \gamma$. On the other hand, any weight $\gamma$ occurring with non-zero multiplicity in $\widehat{\irr}(\lambda)$ with $\gamma \neq \lambda$ must also satisfy
$\gamma < \lambda$, and thus,  $\lambda + \alpha_i \not\leq \gamma$ for $i=1,\dots,n$. Therefore, $\lambda + \alpha_i$ cannot occur as a weight of $E$ which forces $X_{\alpha_i}\cdot v_{\lambda} =0$ for all $i$. 
As a result, the surjective map 
\begin{equation}\label{eqn:ext-weight-spaces}
\begin{aligned}
  \uenvt &\twoheadrightarrow E\\
  X &\mapsto X\cdot v_{\lambda}
\end{aligned}
\end{equation}
factors through the ideal $\widehat{I}_{\lambda}$, and so it follows that $E$ is a quotient of 
${\widehat{\bv}(\lambda) \cong \uenvt/\widehat{I}_{\lambda}}$. 
\end{proof}




\begin{lem}\label{lem:ineq-reduction}
If $M = \Ext^1_{G_1}(\irr(\lambda_i),\irr(\lambda_j))$ is nonzero for $|i-j| \geq 2$, then there exists $\nu \in \bX^+$ such that 
${p\nu \leq \lambda_j - \lambda_i}$. 
\end{lem}
\begin{proof}
The $G/G_1$-module $M$ is nonzero if and only if $M_{p\nu} \neq 0$ for some $\nu \in \bX^+$. By Lemma~\ref{lem:interior-rad1}, we know that there does not exist an extension $E \in M$ which is isomorphic to a quotient of $\bv(\lambda_i)$. 
Hence, there are no extensions $E \in M_{p\nu} = \Ext^1_{G_1T}(\widehat{\irr}(\lambda_i), \widehat{\irr}(\lambda_j - p\nu))$ which occur
as a quotient of $\widehat{\bv}(\lambda_i)$. Thus, by Lemma~\ref{lem:Baby-quotient}, we must have 
$
\lambda_i \leq \lambda_j - p\nu
$
if $M_{p\nu} \neq 0$.
\end{proof}


The vanishing portion of Theorem~\ref{thm:ext1-thm} will follow
 provided that if $|i-j| \geq 2$, then there does not exist any
 $\nu \in \bX^+$ which satisfies
\[
p\nu \leq \lambda_j - \lambda_i. 
\]
\begin{lem}\label{lem:ext-vanishing}
If $0\leq i, j \leq n$ are such that $|i-j| \neq 1$, then 
\[
\Ext_{G_1}^1(\irr(\lambda_i), \irr(\lambda_j)) = 0.
\]
\end{lem}
\begin{proof}
We begin by noting that the $i=j$ case  immediately follows from \cite[Proposition II.12.9]{jantzen}. So we only have to consider the 
case when $|i-j| \geq 2$. 
By the comments immediately preceding this lemma, it suffices to show that there exist no dominant weights $\nu$ such that 
\[
p\nu \leq \lambda_j - \lambda_i,
\]
whenever $|i-j| \geq 2$. Applying \eqref{eqn:weight-formula2}, we compute
\begin{equation*}
\lambda_j - \lambda_i = \begin{cases}
	\epsilon_{j+1} - \epsilon_{i+1} + p(-\varpi_{j+1}+\varpi_{i+1}) & \text{ if $0\leq i,j \leq n-1$}, \\
	\epsilon_{j+1} - \epsilon_{n+1} - p\varpi_{j+1} & \text{ if $0 \leq j \leq n-1$ and $i = n$}, \\
	\epsilon_{n+1} - \epsilon_{i+1} + p\varpi_{i+1} & \text{ if $j = n$ and $0\leq i \leq n-1$}.
    \end{cases} 
\end{equation*}

First suppose that $j>i$, then we can see that 
\begin{equation}\label{eqn:lambda-diferencess}
\lambda_j - \lambda_i < \begin{cases}
   p(-\varpi_{j+1} + \varpi_{i+1}) & \text{ if $j< n$}, \\
   p\varpi_{i+1} & \text{ if $j = n$}. 
   \end{cases}
\end{equation}
The $j=n$ case is now obvious since $\varpi_{i+1}$ is minuscule, and 
thus $p\nu \leq \lambda_n - \lambda_i$ implies
$
p\nu < p\varpi_{i+1}.
$
Hence by \eqref{eqn:very-good-consequence}, $\nu < \varpi_{i+1}$
which is impossible for $\nu \in \bX^+$. On the other hand, for $i < j < n$, 
\[
\lambda_j - \lambda_i < p(-\varpi_{j+1} + \varpi_{i+1}) < p\varpi_{n+1-(j-i)},
\]
where the rightmost inequality comes from the fact that $\varpi_{n+1-(j-i)} = w(-\varpi_{j+1} + \varpi_{i+1})$ for
some $w \in W$. 
Now by applying \eqref{eqn:very-good-consequence} as above, we can see that if $p\nu \leq \lambda_j - \lambda_i$, then 
$\nu < \varpi_{n+1-(j-i)}$, which is also impossible for $\nu \in \bX^+$.

Suppose now that $j < i$. When $i=n$, we can see that $\lambda_j-\lambda_n \not\in p\bX$ so if $p\nu < \lambda_j - \lambda_n$ (and thus $(\lambda_j - \lambda_n) - p\nu \in \Z_{\geq 0}\Phi^+$), then 
\[
(\lambda_j - \lambda_n) - p\nu = (\epsilon_{j+1}-\epsilon_{n+1}) + (-p\varpi_{j+1} - p\nu).
\]
If we set $\gamma = -\varpi_{j+1} - \nu$ and then compare both sides of the preceding equation,
we can deduce that 
$p\gamma \in p\bX\cap \Z\Phi = p\Z\Phi$, where the equality follows from \eqref{eqn:weight-root-intersect}. 
So we can write $p\gamma = \sum_{k=1}^n pc_k\alpha_k$
where $c_k \in \Z$ for all $k$. Moreover, $\gamma \in \Z_{\geq 0}\Phi^+$ since
  \[
   (\epsilon_{j+1}-\epsilon_{n+1}) + p\gamma = \sum_{k=1}^j pc_k\alpha_k + \sum_{k=j+1}^n (pc_k + 1)\alpha_k,
   \]
and thus if $c_k < 0$ for some $k$, then  $pc_k +1 <0$. But this contradicts the assumption that 
$(\lambda_j - \lambda_n) - p\nu \in \Z_{\geq 0}\Phi^+$. 

Hence, 
\[
p\nu \leq -p\varpi_{j+1} < w_0(-p\varpi_{j+1}) = p\varpi_{n-j}
\]
and so by \eqref{eqn:very-good-consequence}, $\nu < \varpi_{n-j}$,
which is impossible for $\nu \in \bX^+$. For $j < i < n$, the same reasoning shows
that if $p\nu \leq \lambda_j - \lambda_i$, then $p\nu \leq p(-\varpi_{j+1} + \varpi_{i+1}) < p\varpi_{i-j-1}$.
This forces $\nu < \varpi_{i-j-1}$, which again is impossible for $\nu \in \bX^+$.
\end{proof}

We now have enough information to complete our $\Ext^1$-calculation. However, before we get to the proof of 
 Theorem~\ref{thm:ext1-thm}, it will be helpful to recall \cite[Lemma 5.1]{andersen-ext}. First, for any 
 $\lambda \in \bX^+$, we introduce the notation
 \[
 \lambda^0 = 2(p-1)\rho + w_0(\lambda).
 \]
 Now suppose ${M = \Ext^1_{G_1}(\irr(\lambda), \irr(\mu))^{(-1)}}$ for some
 $\lambda, \mu \in \bX_1$ and suppose there exists an element $\nu \in \bX$ such that 
 $M_{\nu} \neq 0$, then the aforementioned lemma implies
 \begin{equation*}\label{eqn:andersen-lemma}
 p\nu \leq \mu^0 - \lambda. 
 \end{equation*}
This gives the following lemma. 
 \begin{lem}\label{lem:root-coset}
 Let $M = \Ext^1_{G_1}(\irr(\lambda), \irr(\mu))^{(-1)}$ for some $\lambda, \mu \in \bX_1$. If 
 $\nu, \nu' \in \bX$ are such that $M_{\nu} \neq 0$ and $M_{\nu'} \neq 0$, then 
 $\nu - \nu' \in \Z\Phi$.
 \end{lem}
 \begin{proof}
 By the observation immediately preceding the lemma, we know if $M_{\nu} \neq 0$ and 
 $M_{\nu'} \neq 0$, then both $p\nu$ and $p\nu'$ are in the same root coset since they
 are both comparable to the weight $\mu^0 - \lambda$. Thus by \eqref{eqn:weight-root-intersect},
 \[
 p\nu - p\nu' \in p\bX\cap \Z\Phi = p\Z\Phi.
 \]
Therefore, $\nu - \nu' \in \Z\Phi$ and we are done. 
 \end{proof}
 
 We will also need the following technical lemma.
 \begin{lem}\label{lem:w_1-minimal}
Any $\nu \in (\varpi_1 + \Z\Phi)\cap \bX^+$ satisfies $\varpi_1 \leq \nu$. 
 \end{lem}
 \begin{proof}
Let $G = SL_{n+1}$ with $n\geq 1$, and let 
 $\nu = \varpi_1 + \sum_{i=1}^n a_i\alpha_i$, where $\alpha_1,\dots \alpha_n$ is the set of simple roots and $a_i \in \Z$ is arbitrary for all $i$. 
The lemma will follow if we can show that
 \[
 \nu \in \bX^+ \implies a_i \geq 0 \text{ for all $i$}. 
 \]
(Note that if $n=1$, then the claim is immediate since $\nu = \varpi_1 + a_1\alpha_1 = (2a_1+1)\varpi_1$, where $2a_1 +1 \geq 0$ implies
 $a_1 \geq 0$ because $a_1$ is an integer.)
 
In general,  write $\nu = c_1\varpi_1 + c_2\varpi_2 + \cdots c_n \varpi_n$, and observe that 
$c_1 = 2a_1-a_2+1$, $c_n = -a_{n-1} + 2a_n$ and $c_i = -a_{i-1} + 2a_i -a_{i+1}$ for $i=2, \dots, n-1$. 
Thus, the condition $\nu \in \bX^+$ is given by 
\[
\begin{aligned}
0 & \leq -a_{n-1} + 2a_n \\
 0 & \leq -a_{n-2} + 2a_{n-1} - a_n \\
     & \quad \vdots \\
0  & \leq -a_{n-(k+1)} + 2a_{n-k} - a_{n-(k-1)}\\
    & \quad \vdots \\
0 & \leq -a_1 + 2a_2 - a_3 \\
-1 & \leq 2a_1 - a_2.
\end{aligned}
\]
From the above list of inequalities, we can deduce that $a_n \geq \frac{1}{2}a_{n-1}$, $a_{n-1} \geq \frac{2}{3}a_{n-2}$, and generally, that
\[
\begin{aligned}
    a_{n-k} &\geq \frac{k+1}{k+2}a_{n-(k+1)} \quad \text{for $k = 0, \dots, n-2$,}\\
    a_1 & \geq \frac{-n}{n+1}. 
\end{aligned}
\]

Now since $a_1 \in \Z$ and $-1 < \frac{-n}{n+1} \leq a_1$, we get $a_1 \geq 0$. But then $a_1 \geq 0$ implies
$a_2 \geq \frac{n-1}{n}a_1 \geq 0$, which then implies $a_3 \geq \frac{n-2}{n-1}a_2 \geq 0$. Proceeding in this way, we conclude that 
$a_i \geq 0$ for all $i$. 
 \end{proof}

\begin{proof}[Proof of Theorem~\ref{thm:ext1-thm}]
By Lemma~\ref{lem:ext-vanishing}, we only have to determine 
$${
\Ext^1_{G_1}(\irr(\lambda_i), \irr(\lambda_j))^{(-1)},
}$$
in the case where $|i-j| = 1$. 

Before proceeding we note that a $G$-module $M$ satisfies $M \cong \irr(\varpi_1)$ if and only if the following two conditions hold:
\begin{itemize}
\item [{\bf a)}] $\dim_{\bk} M_{\varpi_1} = 1$,
\item [{\bf b)}] $M_{\nu} = 0$ for any $\nu \in \bX^+$ with $\nu \neq \varpi_1$. 
\end{itemize}
To justify this, first note that if $M \cong \irr(\varpi_1)$, then {\bf a)} and {\bf b)} can be verified by considering the (well known) weight space multiplicities of $\irr(\varpi_1)$ and by recalling the fact that $\varpi_1$ is \emph{minuscule}. Conversely, if $M$ is any module satisfying both conditions, then {\bf b)} 
ensures that $\varpi_1$ is the only possible highest weight, and 
{\bf a)} ensures that it must occur with multiplicity one. In particular, these two conditions force $M$ to have the 
same weight space dimensions as $\irr(\varpi_1)$, and thus, $M \cong \irr(\varpi_1)$. 

We now begin by setting
\[
M = \Ext_{G_1}^1(\irr(\lambda_{i+1}), \irr(\lambda_i))^{(-1)}.
\]
We will first prove that $M$ satisfies {\bf a)}. By \eqref{eqn:cap-factor}, we know that any length two quotient of 
$\wh{\bv}(\lambda_i)$ factors through the module
\[
\mycap^2\, \wh{\bv}(\lambda_i) = \wh{\bv}(\lambda_i)/\rad^2\, \wh{\bv}(\lambda_i),
\]
whose Loewy series is explicitly described in Lemma~\ref{prop:rad1G1T}. In fact, from this description 
 we deduce that there exists (up to isomorphism) precisely one
  $G_1T$-module $E$ which is a quotient of $\wh{\bv}(\lambda_{i+1})$ and which fits into a 
 non-split\footnote{Since $E$ is a quotient of $\wh{\bv}(\lambda_{i+1})$, then it must be indecomposable which forces the short exact sequence to be non-split.}
 short exact sequence of the form 
 \[
 0 \longrightarrow \wh{\irr}(\lambda_i - p\varpi_1)  \longrightarrow E \longrightarrow \wh{\irr}(\lambda_{i+1})  \longrightarrow 0.
 \]
 By \eqref{eqn:ext-weight-space}, it then follows that $\dim_{\bk} M_{\varpi_1} \geq 1$. 
 
 On the other hand, suppose
$E'$ is an arbitrary $G_1T$-module which fits into a short exact sequence as above. Now since we can deduce from
\eqref{eqn:lambda-diferencess} that
\[
\lambda_{i+1} \not\leq \lambda_{i} - p\varpi_1,
\]
then by Lemma~\ref{lem:Baby-quotient}, $E'$ is also a quotient of $\wh{\bv}(\lambda_{i+1})$. Hence
  $E' \cong E$, and therefore, $\dim_{\bk} M_{\varpi_1} = 1$. So we have verified {\bf a)}.

We will now verify {\bf b)} by contradiction. Suppose there exists 
$\nu \in \bX^+$ with $\nu \neq \varpi_1$ and $M_{\nu} \neq 0$. By Lemma~\ref{lem:root-coset}, we know that
$\nu \in (\varpi_1 + \Z\Phi)\cap \bX^+ $, and hence, $\varpi_1 < \nu$ by Lemma~\ref{lem:w_1-minimal}. 
In addition, we must also have
$p\nu \not\leq \lambda_i - \lambda_{i+1}$ since $p\varpi_1 \not\leq \lambda_i - \lambda_{i+1}$. 
So by Lemma~\ref{lem:Baby-quotient} there exists a quotient $E'$ of 
$\wh{\bv}(\lambda_{i+1})$ which fits into a short exact sequence of the form   
 \[
 0 \longrightarrow \wh{\irr}(\lambda_i - p\nu)  \longrightarrow E' \longrightarrow \wh{\irr}(\lambda_{i+1})  \longrightarrow 0.
 \]
 
In particular, $\wh{\irr}(\lambda_i-p\nu)$ must occur as a factor of 
$\radl_1\, \bv(\lambda_{i+1})$. But from Lemma~\ref{prop:rad1G1T} we can see that there are no such factors (i.e. there are no factors of the form $\wh{\irr}(\lambda_i -p\nu)$ with $\nu > \varpi_1$). It follows that 
$M_{\nu} =0$ and we have reached a contradiction. 

Similarly, if we set
\[
N = \Ext_{G_1}^1(\irr(\lambda_{i}), \irr(\lambda_{i+1}))^{(-1)},
\]
then by the same reasoning as above, we get $N \cong \irr(\varpi_1)^*$. 
\end{proof}

\section{The Loewy series for $\widehat{\bv}(\lambda_i + p\nu)$ and $\wh{\bv}'(\lambda_i + p\nu)$}\label{sec:verma-radical-layers}

In this section we will determine Loewy series for
$\widehat{\bv}(\lambda_i+p\nu)$ and $\wh{\bv}'(\lambda_i+p\nu)$. We will also deduce the Loewy lengths and establish the 
rigidity of these modules. 
We now begin by considering the easier problem involving $\bv(\lambda_i)$ and $\bv'(\lambda_i)$. 
\begin{lem}\label{lem:radical-subdivision}
Let $n\geq 2$, then for $0\leq i \leq n$ and $j\geq 0$, 
\begin{equation}\label{eqn:parity1}
[\irr(\lambda_k)] \leq [\radl_j\, \bv(\lambda_i)] \implies k \equiv i + j \mod 2.
\end{equation}
Moreover, for $0 \leq i \leq n-1$, 
\begin{equation}\label{eqn:sum-formula-I}
\nolF_I^j(\lambda_i) + \rad^1\,\nolF_I^{j-1}(\lambda_i)  \subseteq \rad^j\, \bv(\lambda_i), \quad  \radl_j\, \bv(\lambda_i) = \radl_0\, \olF_I^j(\lambda_i) \oplus \radl_1\, \olF^{j-1}_I(\lambda_i).
\end{equation}
Also,  for $1\leq i \leq n$,
\begin{equation}\label{eqn:sum-formula-J}
\nolF_J^j(\lambda_i) +  \rad^1\,\nolF_J^{j-1}(\lambda_i) \subseteq \rad^j\, \bv(\lambda_i), \quad \radl_j\, \bv(\lambda_i) = \radl_0\, \olF_J^j(\lambda_i) \oplus \radl_1\, \olF^{j-1}_J(\lambda_i).
\end{equation}
(We set $ \olF^{-1}_I(\lambda_i) = 0$ and $\olF^{-1}_J(\lambda_i) =0$.)
\end{lem}
\begin{proof}
As in the proof of Lemma~\ref{lem:interior-rad1}, we will proceed by induction on $n\geq 2$. 
The base case again follows from the explicit formulas given in \cite[Theorems 2.4-2.5]{xi1999}. 
Suppose $n > 2$ and assume the statement of the lemma holds for all $SL_{r+1}$ with $2 \leq r < n$. 
The argument in the proof of Proposition~\ref{prop:bv_multiplicity} implies that the statement also holds for the Levi factor 
$L_I$ with $\bv_I(\lambda_i)$ and $0 \leq i \leq n-1$ (respectively $L_J$ with $\bv_J(\lambda_i)$ 
and $1 \leq i \leq n$).  

For simplicity, let us begin by fixing $i \in \{0,\dots, n-1\}$. The inductive hypothesis gives 
\[
\radl_j\, \bv_I(\lambda_i) = \bigoplus_{0 \leq k \leq n-1} \irr_I(\lambda_k)^{\oplus m^j_{ik}},
\]
where $m^j_{ik} \neq 0$ implies $k \equiv i + j \mod 2$. 
Thus, 
\[
\olF_I^j(\lambda_i) = \bigoplus_{0 \leq k \leq n-1} \bm_I(\lambda_k)^{\oplus m^j_{ik}}.
\]

By Proposition~\ref{prop:parabolic-coinduction1}, 
\begin{equation*}\label{eqn:radical-hypothesis}
\begin{aligned}
  \radl_0\, \olF_I^j(\lambda_i) = \bigoplus_{0 \leq k \leq n-1} \irr(\lambda_k)^{\oplus m^j_{ik}}, 
  \quad \radl_1\, \olF_I^j(\lambda_i) = \bigoplus_{0 \leq k \leq n-1} \irr(\lambda_{k+1})^{\oplus m^j_{ik}}.
\end{aligned}
\end{equation*}
From the inductive hypothesis, we can see that \eqref{eqn:parity1} will hold on the factors of 
$\radl_j\, \bv(\lambda_i)$, provided we verify \eqref{eqn:sum-formula-I}. 
We will proceed by induction on $j\geq 0$. The base case, $j=0$, is obvious since 
${\radl_0\, \bv(\lambda_i) = \irr(\lambda_i)}$. Also, the $j=1$ case follows from Lemma~\ref{lem:interior-rad1}. 
Now assume $j\geq 2$ and that
 \eqref{eqn:sum-formula-I} holds for $0 \leq l < j$. 
 
 The inductive hypothesis for $j$ gives $\nolF_I^{j-1}(\lambda_i) \subseteq \rad^{j-1}\, \bv(\lambda_i)$, and thus,  \\
${\rad^1\, \nolF_I^{j-1}(\lambda_i) \subseteq  \rad^{j}\, \bv(\lambda_i)}$. 
Now if we apply the inductive hypothesis for $L_I$ and reason as we did in the portion of the proof of Lemma~\ref{lem:interior-rad1}
between \eqref{eqn:rad0-olF1} and \eqref{eqn:F2-rad-inclusion}, we can deduce
 \[
 \nolF_I^j(\lambda_i) \subseteq  \rad^1\, \nolF_I^{j-1}(\lambda_i) \subseteq  \rad^{j}\, \bv(\lambda_i),
 \]
 and hence, the first claim of \eqref{eqn:sum-formula-I}. 
Moreover, by imitating the arguments immediately following \eqref{eqn:F2-rad-inclusion}, we can also show that 
$\radl_j\, \bv(\lambda_i)$  is the head of the module
 \[
 M = \frac{\rad^j\, \bv(\lambda_i)}{\rad^1\, \nolF_I^j(\lambda_i)}, 
 \]
which fits into a short exact sequence of the form
 \[
 0 \longrightarrow \radl_0\, \olF_I^j(\lambda_i) \longrightarrow M \longrightarrow \radl_1\, \olF_I^{j-1}(\lambda_i)
   \longrightarrow 0.
 \]
But now, by the inductive hypothesis and \eqref{eqn:radical-hypothesis}, we can see that every 
factor $\irr(\lambda_k)$ of $M$ occurring with non-zero multiplicity must satisfy $k \equiv i + j \mod 2$. 
In particular, if $\irr(\lambda_s)$ and $\irr(\lambda_t)$ are two non-zero factors of $M$, then 
$|s-t| \neq 1$. Thus, Theorem~\ref{thm:ext1-thm} implies the preceding short exact sequence is split, and
hence \eqref{eqn:sum-formula-I} holds for all $j\geq 0$. 

So we have verified \eqref{eqn:parity1} and \eqref{eqn:sum-formula-I} for 
 $0 \leq i \leq n-1$ and $j\geq 0$. On the other hand, if we replace $I$ with $J$ and fix 
any $i \in \{1,\dots, n\}$, then the same argument as above also verifies \eqref{eqn:parity1} and
\eqref{eqn:sum-formula-J} for all $j\geq 0$. 
\end{proof}

Before getting to the main results of this section, we recall a simple combinatorial identity obtained 
from \emph{Pascal's triangle}. 
Namely, for any $0 \leq j \leq n$ and $0 \leq i \leq j$, 
\begin{equation}\label{eqn:pascal-triangle}
{n \choose j} = \sum_{0 \leq k \leq i} {i \choose k}{n-i \choose j-k},
\end{equation}
 where we assume ${n \choose j} =0$ unless ${0\leq j \leq n}$. 

\begin{prop}\label{prop:radical-layers}
Let $n\geq 1$, then for $0\leq i \leq n$ and 
$j\geq 0$,
\begin{equation}\label{eqn:radical-formula}
 \radl_j\, \bv(\lambda_i) = \socl_{j+1}\, \bv'(\lambda_i) = \bigoplus_{0 \leq k \leq i} \irr(\lambda_{i+j - 2k})^{\oplus {i \choose k}{n-i \choose j-k}}.
\end{equation}
In particular, ${\ell\ell(\bv'(\lambda_i))=\ell\ell(\bv(\lambda_i)) = n+1}$ and $\radl_j\, \bv(\lambda_i)$ has precisely ${n \choose j}$ factors. 
\end{prop}
\begin{proof}
By \eqref{eqn:socle-radical-duality}, we are reduced to determining the radical layers of $\bv(\lambda_i)$. As usual, we will prove \eqref{eqn:radical-formula} by induction on $n\geq 1$. The base case, $n=1$, is 
trivial. Now assume by induction that the formula holds for  $SL_{r+1}$ with $1 \leq r < n$ and apply this to $L_I$. If we fix
$i \in \{0,\dots, n-1\}$, then by Lemma~\ref{lem:radical-subdivision} and Proposition~\ref{prop:parabolic-coinduction1},
\[
\begin{aligned}
\radl_j\, \bv(\lambda_i) &= \radl_0\, \olF_I^j(\lambda_i) \oplus \radl_1\, \olF^{j-1}_I(\lambda_i)\\
                 &= \bigoplus_{0 \leq k \leq i} \irr(\lambda_{i+j - 2k})^{\oplus {i \choose k}{n-i-1 \choose j-k}}
                  \oplus \bigoplus_{0 \leq k \leq i} \irr(\lambda_{i+j - 2k})^{\oplus {i \choose k}{n-1-i \choose j-k-1}}\\
                  &=  \bigoplus_{0 \leq k \leq i} \irr(\lambda_{i+j - 2k})^{\oplus {i \choose k}\left({n-i-1 \choose j-k} +{n-1-i \choose j-k-1}\right)}\\
                  &= \bigoplus_{0 \leq k \leq i} \irr(\lambda_{i+j - 2k})^{\oplus {i \choose k}{n-i \choose j-k}}.
\end{aligned}
\]

Similarly, we can verify \eqref{eqn:radical-formula} for $i \in \{1,\dots,n\}$ by applying the inductive hypothesis to 
$L_J$. 
\end{proof}

Using the same methods as above, we can determine the radical layers of $\widehat{\bv}(\lambda_i + p\nu)$ (or equivalently the 
socle layers of ${\wh{\bv}'(\lambda_i + p\nu)}$ by \eqref{eqn:socle-radical-duality}) for $0 \leq i \leq n$. 

For any $i \leq j$, set 
 \[
 [i, j] =\{i, i+1,\dots, j\}, 
 \]
  and for any subset $X \subseteq [1,n+1]$, we define
 \[
 \epsilon_X = \sum_{k \in X}\, \epsilon_k,
 \]
 where $\epsilon_{\emptyset} = 0$. 
 The $\lambda_0$ and $\lambda_n$ formulas are now easily obtained by applying Lemma~\ref{lem:radical-subdivision} and Proposition~\ref{prop:parabolic-coinduction1}. In particular, for any $\nu \in \bX$, 
 \begin{equation}\label{eqn:lambda0-rad}
 \radl_j\, \widehat{\bv}(\lambda_0+p\nu) = \bigoplus_{\{X\subseteq [2, n+1] \, \mid \, |X| = j\}} \widehat{\irr}(\lambda_j  + p\nu+ p\epsilon_X),
 \end{equation}
  \begin{equation}\label{eqn:lambdan-rad}
 \radl_j\, \widehat{\bv}(\lambda_n+p\nu) = \bigoplus_{\{X\subseteq [1,n]\, \mid\, |X| = j\}} \widehat{\irr}(\lambda_{n-j} + p\nu - p\epsilon_X).
 \end{equation}
 
 To handle the $0 <i < n$ case, we introduce the subsets $I_i = \{\epsilon_1-\epsilon_2, \dots, \epsilon_i-\epsilon_{i+1}\}$. We then apply
 \eqref{eqn:lambdan-rad} to $\widehat{\bv}_{I_i}(\lambda_i)$ and get
 \[
  \radl_k\, \widehat{\bv}_{I_i}(\lambda_i) = \bigoplus_{\{X\subseteq [1,i]\, \mid \, |X| = k\}} \widehat{\irr}_{I_i}(\lambda_{i-k} - p\epsilon_X),
 \]
 with $0\leq k \leq i$ (it is zero otherwise). 
The radical layers of $\widehat{\bv}(\lambda_i)$ are computed by repeatedly applying Lemma~\ref{lem:radical-subdivision} and Proposition~\ref{prop:parabolic-coinduction1} to each radical layer of 
 ${\widehat{\bv}_{I_r}(\lambda_i)}$ for $i< r \leq n$. In particular, applying this procedure to each
${\widehat{\irr}_{I_i}(\lambda_{i-k} - p\epsilon_X)}$ produces an object ``$M_{i-k, X}$", whose non-zero radical layers are given by
\[
\radl_s\, M_{i-k, X} = \bigoplus_{\{Y \subseteq [i+2,n+1]\, \mid \, |Y| = s\}} \widehat{\irr}(\lambda_{i-k +s} -p\epsilon_X + p\epsilon_Y),
\]
with $0 \leq s \leq n-i$. The radical layers of $\widehat{\bv}(\lambda_i)$ are actually built out of various ``$k$-shifted" copies of 
$\radl_s\, M_{i-k, X}$, where we have 
\[
\radl_s\, M_{i-k, X} \subseteq \radl_{s+k}\, \widehat{\bv}(\lambda_i).
\]
Altogether, we get
\[
\radl_j\, \widehat{\bv}(\lambda_i) = \bigoplus_{k=0}^{i} \bigoplus_{\{X\subseteq [1,i]\, \mid \, |X| = k\}} \radl_{j-k}\, M_{i-k, X}.
\]
Thus, we have proven the following theorem.
\begin{thm}\label{thm:g1t-radical-layers}
Let $n\geq 1$, then for $0\leq i \leq n$, $\nu \in \bX$ and 
any $j\geq 0$,
\begin{align*}
\radl_j\, &\widehat{\bv}(\lambda_i + p\nu) = \socl_{j+1}\, \widehat{\bv}'(\lambda_i + p\nu) \\
 &=\bigoplus_{k=0}^{i} \bigoplus_{\{ (X,Y) \,\mid \, |X| = k, \, |Y| = j-k,\, X \subseteq [1,i], 
Y \subseteq [i+2,n+1] \}} \widehat{\irr}(\lambda_{i+j -2k} +p\nu -p\epsilon_X + p\epsilon_Y).
\end{align*}
\end{thm}
\begin{rmk}
Compare with \cite[Theorem, p. 2]{nm}.
\end{rmk}

 The arguments used in Proposition~\ref{prop:radical-layers} and Theorem~\ref{thm:g1t-radical-layers}, can also 
 be adapted to compute $\socl_j\, \wh{\bv}(\lambda_i)$ for $j \geq 1$ (or equivalently $\radl_j\, \wh{\bv}'(\lambda_i)$ for $j\geq 0$). 
 \begin{prop}\label{prop:rigidity}
 Let $n\geq 1$, then for $0\leq i \leq n$, $\nu \in \bX$ and any
$j\geq 1$,
 \[
 \socl_j\,\wh{\bv}(\lambda_i + p\nu) \cong  \radl_{n+1-j}\,\wh{\bv}(\lambda_i + p\nu), \quad
  \socl_j\,\wh{\bv}'(\lambda_i + p\nu) \cong  \radl_{n+1-j}\,\wh{\bv}'(\lambda_i + p\nu).
 \] 
In particular, $\wh{\bv}(\lambda_i + p\nu)$ and  $\wh{\bv}'(\lambda_i + p\nu)$ are rigid modules. 
 \end{prop}

\section{The Loewy series for $\widehat{\Inj}(\lambda_i + p\nu)$}\label{sec:radical-layers}
We will now show that the results for $\widehat{\bv}(\lambda_i + p\nu)$ from the preceding sections enable us to adapt 
 the arguments from \cite{ak1989} to our setting and determine Loewy series for the $\widehat{\Inj}(\lambda_i + p\nu)$.  From now on, we will additionally assume that $p$ is large enough so that the following 
conjecture holds. 
\begin{conj}\label{conj:inj-length}
Let $n \geq 1$, then for $0 \leq i \leq n$ and $\nu \in \bX$, 
$\ell\ell(\wh{\Inj}(\lambda_i+p\nu)) = 2n+1$. 
\end{conj}
\begin{rmk}\label{rmk:large-p}
This is known to hold for extremely large $p$ by \cite[Theorem, p. 10]{nm}. 
 \end{rmk}
 
The remainder of the section will be devoted to proving the following theorem. 
\begin{thm}\label{thm:singular-reciprocity}
Suppose Conjecture~\ref{conj:inj-length} holds. Let $n\geq 1$, then for $0\leq i \leq n$ and $\nu \in \bX$,
$\wh{\Inj}(\lambda_i + p\nu)$ is rigid
 and for any $j\geq 0$
  \begin{equation}\label{eqn:singular-reciprocity}
  [\radl_j\, \wh{\Inj}(\lambda_i + p\nu)] =  \sum_{\mu \in \bX}\sum_{k=0}^j [\radl_k\, \wh{\bv}(\mu):\wh{\irr}(\lambda_i+p\nu)][\radl_{j-k}\, \wh{\bv}(\mu)].
  \end{equation}
\end{thm}
\begin{rmk}
Obviously, ${[\radl_k\, \wh{\bv}(\mu):\wh{\irr}(\lambda_i +p\nu)] = 0}$ unless ${\mu = \lambda_t + p\eta}$ for some $0 \leq t \leq n$ and $\eta \in \bX$.
So by combining the preceding theorem with Theorem~\ref{thm:g1t-radical-layers}, we can 
completely determine the Loewy series of the
$\wh{\Inj}(\lambda_i + p\nu)$. 
\end{rmk}
For the remainder of the section, we will fix $I, J \subset S$ as in \S\ref{sec:initial-results}.
 Let us first observe that from the identities \eqref{eqn:dual-filtrations}, \eqref{eqn:duality-G1T} and \eqref{eqn:socle-radical-duality}, 
 it can be shown that \eqref{eqn:singular-reciprocity} holds for all $j\geq 0$
 if and only if
 \begin{equation}\label{eqn:singular-reciprocity2}
 [\soc^j\, \wh{\Inj}(\lambda_i + p\nu)] = \sum_{\mu \in \bX}\sum_{k=1}^j [\socl_k\, \wh{\bv}'(\mu):\wh{\irr}(\lambda_i +p\nu)]
 				[\soc^{j+1-k}\, \wh{\bv}'(\mu)]
 \end{equation}
 for all $j\geq 1$ (compare with \cite[Theorem 7.2(ii)]{ak1989}). 
 
 It turns out that the preceding identity is always ``partially" true by the following lemma (adapted from \cite[Proposition~3.7]{ak1989}). 
 \begin{lem}\label{lem:ak-prop3.7}
 For any $\lambda \in \bX$ and $j\geq 1$,
   \begin{equation*}
  [\soc^j\, \wh{\Inj}(\lambda)] \leq \sum_{\mu \in \bX}\sum_{k=1}^j [\socl_k\, \wh{\bv}'(\mu):\wh{\irr}(\lambda)]
 				[\soc^{j+1-k}\, \wh{\bv}'(\mu)].
 \end{equation*}
 \end{lem}
 \begin{proof}
We first note that the Lemmas occurring in
  \cite[3.5 and 3.6]{ak1989} can be adapted to our setting. 
   This is because their proofs essentially consist of the same types of arguments employed in the proof of \cite[Proposition II.11.2]{jantzen}, as well as certain 
    general results
  on socle filtrations of modules, and on the basic properties of
  $\wh{\bv}'(\lambda)$ (e.g. the highest weight structure). In particular, there is no dependence on the $p$-regularity of $\lambda \in \bX$, or even on the prime $p$.  The proof of our result follows by applying  the more general versions of these lemmas to imitate the proof of \cite[Proposition~3.7]{ak1989}.
 \end{proof}
 
 Next, we observe that results from \S\ref{sec:verma-radical-layers} imply the following analogue to \cite[Lemma 7.1]{ak1989}. 
 \begin{lem}\label{lem:ak-lem7.1}
Let $n\geq 1$, then for $0\leq i \leq n$ and $\nu \in \bX$,
\begin{multline*}
\sum_{\mu \in \bX}\sum_{k=1}^j [\socl_k\, \wh{\bv}'(\mu):\wh{\irr}(\lambda_i + p\nu)][\soc^{j+1-k}\, \wh{\bv}'(\mu)] = \\ 
\sum_{\mu \in \bX}\sum_{k=1}^j [\socl_k\, \wh{\bv}'(\mu):\wh{\irr}(\lambda_i+p\nu)][\mycap^{j+k-n-1}\, \wh{\bv}'(\mu)],
\end{multline*}
for all $j \geq 1$. 
 \end{lem}

 \begin{proof}[Proof of Theorem~\ref{thm:singular-reciprocity}]
 To verify \eqref{eqn:singular-reciprocity2} (which is equivalent to \eqref{eqn:singular-reciprocity}),  we will proceed as in the proof of \cite[Theorem 7.2]{ak1989}. Namely, observe that
 $${ }^{\tau} \wh{\Inj}(\lambda_i + p\nu) \cong  \wh{\Inj}(\lambda_i + p\nu)$$ implies 
 \[
  [\soc^j\, \wh{\Inj}(\lambda_i + p\nu)] = [\mycap^j\,  \wh{\Inj}(\lambda_i + p\nu)]
  \]
  for all $j\geq 1$.
Applying Lemma~\ref{lem:ak-prop3.7}, Lemma~\ref{lem:ak-lem7.1}, and Proposition~\ref{prop:rigidity}, we get
\begin{align*}
&[\soc^j\, \wh{\Inj}(\lambda_i + p\nu)] + [\mycap^{2n+1-j}\,\wh{\Inj}(\lambda_i + p\nu)] \\
   &= [\soc^j\, \wh{\Inj}(\lambda_i + p\nu)] + [\soc^{2n+1-j}\, \wh{\Inj}(\lambda_i + p\nu)]\\
   &\leq \sum_{\mu \in \bX}\sum_{k} [\socl_k\, \wh{\bv}'(\mu):\wh{\irr}(\lambda_i+p\nu)]\big(
 				[\soc^{j+1-k}\, \wh{\bv}'(\mu)] + [\soc^{2n+1-j-k}\, \wh{\bv}'(\mu)]\big) \\
    &= \sum_{\mu \in \bX}\sum_{k} [\socl_k\, \wh{\bv}'(\mu):\wh{\irr}(\lambda_i+p\nu)] \\
     &\quad\quad \quad \quad\times  \big(
 				[\soc^{j+1-k}\, \wh{\bv}'(\mu)] +[\wh{\bv}'(\mu)]- [\mycap^{j+k-n-1}\, \wh{\bv}'(\mu)]\big) \\
    &= \sum_{\mu \in \bX}\sum_{k} [\socl_k\, \wh{\bv}'(\mu):\wh{\irr}(\lambda_i+p\nu)][\wh{\bv}'(\mu)]\\
    &= [\wh{\Inj}(\lambda_i + p\nu)] \quad\quad \text{(by \cite[Proposition II.11.4]{jantzen})}. 
\end{align*}

Combining this with Conjecture~\ref{conj:inj-length} and \eqref{eqn:socle-cap}, then gives
\[
 [\wh{\Inj}(\lambda_i + p\nu)] = [\soc^j\, \wh{\Inj}(\lambda_i + p\nu)] + [\mycap^{2n+1-j}\,\wh{\Inj}(\lambda_i + p\nu)],
\]
and hence the rigidity result follows. We are also forced to have both 
\begin{align*}
[\soc^j\, \wh{\Inj}(\lambda_i + p\nu)] &=  \sum_{\mu \in \bX}\sum_{k} [\socl_k\, \wh{\bv}'(\mu):\wh{\irr}(\lambda_i+p\nu)]
 				[\soc^{j+1-k}\, \wh{\bv}'(\mu)],\\
 [\soc^{2n+1-j}\, \wh{\Inj}(\lambda_i + p\nu)] &= \sum_{\mu \in \bX}\sum_{k} [\socl_k\, \wh{\bv}'(\mu):\wh{\irr}(\lambda_i+p\nu)]
 			  [\soc^{2n+1-j-k}\, \wh{\bv}'(\mu)],
\end{align*}
by Lemma~\ref{lem:ak-prop3.7}. Therefore, \eqref{eqn:singular-reciprocity2} must also hold. 
 \end{proof}

\end{document}